\documentclass[11pt,a4paper]{amsart}
\usepackage{amsmath,caption,booktabs,lipsum}
\numberwithin{equation}{section}
\usepackage{amsthm}
\usepackage{pdfpages}
\usepackage{relsize}
\usepackage{amsfonts}
\usepackage{amssymb}
\usepackage{mathrsfs}
\usepackage{tikz}
\usepackage{caption}
\usepackage{environ}
\usepackage{elocalloc}
\usepackage{url}
\usepackage{caption}
\usepackage{graphicx}
\usepackage[backend = bibtex, style=numeric-comp,doi=false,isbn=false,url=false]{biblatex}
\usepackage{CJKutf8}
\usepackage[normalem]{ulem}
\usepackage{xcolor}
\usepackage{etoolbox}
\usepackage{caption}
\usepackage{mathtools}
\usepackage{dcpic}
\usepackage{tikz-cd}
\usepackage[bottom]{footmisc}
\usepackage{hyperref}
\usepackage{enumitem}
\usepackage{stmaryrd}
\usetikzlibrary{positioning}
\usetikzlibrary{arrows,chains,positioning,scopes,quotes}
\usetikzlibrary{decorations.markings}
\newcommand{\xn}{X\times N_0}
\newcommand{\yn}{Y\times N_f}
\newcommand{\hgd}{h_{\operatorname{glued}}}
\newcommand{\htt}{h_{\operatorname{transit}}}
\newcommand{\hpp}{h_{\operatorname{prepare}}}
\newcommand{\hmt}{h_{\operatorname{meet}}}

\newcommand{\supp}{\operatorname{Supp}}
\newcommand{\reg}{{\operatorname{Reg}}}
\newcommand{\ai}{\alpha}

\newcommand{\sing}{{\operatorname{Sing}}}

\newcommand{\mff}{\mathcal{F}}

\newcommand{\es}{\emptyset}
\newcommand{\urm}{U_r^{\textnormal{mollified}}}
\newcommand{\urmh}{\hat{U}_r^{\textnormal{mollified}}}

\newcommand{\Ga}{\Gamma}
\newcommand{\ga}{\gamma}

\newcommand{\de}{\delta}

\newcommand{\e}{\epsilon}

\newcommand{\lam}{\lambda}

\newcommand{\si}{\sigma}
\newcommand{\Si}{\Sigma}

\newcommand{\rh}{\rho}

\newcommand{\cms}{\operatorname{comass}}

\newcommand{\bcs}{\prescript{}{\pd}{\#}}

\newcommand{\ms}{\mathbf{M}}

\newcommand{\cd}{\cdots}

\newcommand{\ug}{U_{\operatorname{gluing}}}
\newcommand{\s}{\subset}

\newcommand{\cp}{^\complement}

\newcommand{\kn}{\ker^\perp \ed\Pi}

\newcommand{\ov}[1]{\overline{#1}}
\newcommand{\no}[1]{\left\lVert#1\right\rVert}
\newcommand{\cu}[1]{\left\llbracket#1\right\rrbracket}

\newcommand{\dvol}{\operatorname{dvol}}
\newcommand{\ed}{{\operatorname{d}}}

\newcommand{\hkp}{h_{\ker \ed\Pi}}
\newcommand{\hkn}{h_{\ker^\perp \ed\Pi}}
\newcommand{\du}{^\ast}
\newcommand{\pf}{_\ast}

\newcommand{\ka}{\kappa}
\newcommand{\pxn}{\pi_{X\times N_0}}
\newcommand{\pyn}{\pi_{Y\times N_f}}
\newcommand{\dvx}{\dvol_{X\times N_0}^{g_N}}
\newcommand{\dvy}{\dvol_{Y\times N_f}^{g_N}}
\newcommand{\m}{^{-1}}

\newcommand{\cz}{\C^{d-\dim N}/\Z^{2(d-\dim N)}}

\newcommand{\imb}{i(M\setminus B)}
\newcommand{\imbb}{i'(M'\setminus B')}
\newcommand{\w}{\wedge}

\newcommand{\pd}{\partial}

\newcommand{\mfr}{{\mathbf{r}}}
\newcommand{\tvu}{(T\#V)|_U}

\newcommand{\N}{\mathbb{N}}
\newcommand{\R}{\mathbb{R}}
\newcommand{\Z}{\mathbb{Z}}

\newcommand{\Q}{\mathbb{Q}}
\newcommand{\C}{\mathbb{C}}

\newcommand{\ur}{B_{\mfr}(\tvu)}

\newcommand{\id}{\textnormal{id}}

\makeatletter
\def\thm@space@setup{%
	\thm@preskip=0.2cm plus 0cm minus 0cm
	\thm@postskip=\thm@preskip 
}
\makeatother
\theoremstyle{plain}
\newtheorem{thm}{Theorem}
\newtheorem{exam}{Example}[section]
\newtheorem{lem}[exam]{Lemma}

\newtheorem{fact}[exam]{Fact}

\newtheorem{assump}[exam]{Assumption}

\newtheorem{defn}[exam]{Definition}

\theoremstyle{definition}
\newtheorem{conj}{Conjecture}

\title[Fractal singular sets II]{On a conjecture of Almgren II: area-minimizing submanifolds with fractal singular sets on almost any manifold}
\author{Zhenhua Liu}
\dedicatory{Dedicated to Xunjing Wei}
\numberwithin{equation}{subsection}
\begin{document}
	\setlength{\abovedisplayskip}{5pt}
	\setlength{\belowdisplayskip}{5pt}
	\setlength{\abovedisplayshortskip}{5pt}
	\setlength{\belowdisplayshortskip}{5pt}
	\maketitle\vspace{-3em}
	\begin{abstract}
		This paper is the second in a two-part solution to Almgren's conjecture on the existence of area-minimizing submanifolds with fractal singular sets. In part one, we construct area-minimizing submanifolds with fractal singular sets on certain special manifolds. Here we continue our work and show that area-minimizing submanifolds with fractal singular sets exist on almost any smooth manifold. 
	\end{abstract}
	\section{Introduction}
	The general problem of finding area-minimizing representatives in integral homology classes is solved by Federer and Fleming (\cite{FF}):
	\begin{quote}
		\emph{Every integral homology class on a compact Riemannian manifold admits an area-minimizing representative.}
	\end{quote}
	In other words we can always find a representative of an integral homology class that has least area among all representatives of the same homology class. The representative in the above result is found in the category of integral currents, which roughly speaking is the closure of the set of algebraic topological polyhedron chains under Whitney flat topology \cite{FF}.
	
	Calling the area-minimizing representatives submanifolds was justified due to the following theorem:
	\begin{quote}
		\emph{An area-minimizing representative of an integral homology class is a smooth minimal submanifold counted with integer multiplicities  outside of a singular set that is codimension $2$ countably rectifiable with respect to the submanifolds.}
	\end{quote}Almgren \cite{FA} first established that the singular set is of Hausdorff codimension $2,$ while the hypersurface case dated to  Federer \cite{HFts}. De Lellis and Spadaro simplified Almgren's proof in \cite{DS1,DS2,DS3}. Simultaneously and independently, De Lellis-Minter-Skoborotova 
	(\cite{DMS1,DMS2,DMS3}) and Wickramasekera-Krummel (\cite{BK1,BK2,BK3})  proved the above tour de force result. 
	
	Thus, we will use interchangeably the terms area-minimizing representatives of homology classes and area-minimizing submanifolds. Similar results hold for area-minimizing representatives of finite coefficient homology (\cite{DMSmod} and the references therein.)
	
	The above results on the singular sets of area-minimizing submanifolds are sharp, as singular sets do appear in very natural settings \cite[Section 4]{MR0168727}:
	\begin{quote}
		\emph{Complex algebraic subvarieties, including singular subvarieties, are area-minimizing in the Fubini-Study metric on complex projective spaces.}
	\end{quote}
	In the same vein, many singular complex hyperquadrics are area-minimizing in every finite coefficient homology class they belong to (\cite{FMcalv}). 
	
	To the author's knowledge, like complex algebraic subvarieties, all explicitly known examples of area-minimizing representatives of integral and finite coefficient homologies are subanalytic subvarieties, and so are their singular sets. This is in sharp contrast with the above state of the art on the singular sets of area-minimizing currents, which gives us only the rectifiability of singular sets.

	Already noticing this gap between general theories and known examples in the 1980s, Professor Frederick Almgren has raised the following question as Problem 5.4 in \cite{GMT}, (to quote him), 
	\begin{conj}\label{ca}
		\emph{"Is it possible for the singular set of an area-minimizing integer (real) rectifiable current to be a Cantor type set with possibly non-integer Hausdorff dimension?"}
	\end{conj}
	The ground-breaking work of Leon Simon \cite{LSfr} answered an analogue of the above question in the setting of stable stationary varifolds and inspired this manuscript. It is widely believed that the examples constructed in \cite{LSfr} are also area-minimizing, thus settling Almgren's conjecture in the hypersurface case. The main theorems of part one of this manuscript \cite{ZLa} answer the above conjecture by Almgren affirmatively in other cases and the answers are sharp in terms of dimensions of the singular sets.
	\begin{thm}\label{thmd}(\cite{ZLa})
		For any integer $d\ge 2$ and any nonnegative real number $0\le\ai\le d-2,$ there exists a smooth compact $(d+3)$-dimensional Riemannian manifold $M^{d+3}$, and  a $d$-dimensional area-minimizing integral current $\Si$ on $M$ such that
		\begin{enumerate}
			\item The singular set of $\Si$ is of Hausdorff dimension $\ai$.
			\item A $d$-dimensional smooth calibration form $\phi$ on $M$ calibrates $\Si.$
		\end{enumerate}
	\end{thm}
	In other words, in \cite{ZLa}, we proved that there exist  area-minimizing submanifolds with fractal singular sets on certain special manifolds. In this manuscript, as announced in \cite[Section 5.6]{ZLa},  we vastly extend the constructions in \cite{ZLa} and show that 
	\begin{thm}\label{thma}
		Area-minimizing submanifolds with fractal singular sets exist on almost any smooth manifold. 
	\end{thm}
	In other words, area-minimizing submanifolds with fractal singular sets are natural and abundant.
	
	As a by-product of the proof of Theorem \ref{thma}, inspired by Zhang's work \cite{YZa,YZj,YZt}, we obtain a connected sum theorem in the category of calibrated currents, which might be of independent interest.
	\begin{thm}\label{thmcs}
		The category of calibrated multiplicity $1$ integral currents is closed under the connected sums of chains inside the same ambient manifold, the standard connected sum of ambient manifolds and the boundary connected sums of ambient manifolds. 
	\end{thm}
	Here the class of multiplicity $1$ currents \cite{LScy} is defined in Definition \ref{defnmo}.

	To state the results more precisely, let us start with some basic assumptions and definitions that we will use throughout this manuscript.
	
	For a $d$-dimensional stationary varifold $T$, which is a much weaker notion than being area-minimizing,  the Almgren stratification of $ T$ (\cite{BWst}) is an ascending chain of closed subsets of the support of  $T$
	\begin{align*}
		\mathcal{S}_0T\s\mathcal{{S}}_1T\s\cd\s\mathcal{S}_dT=\supp T,\end{align*}
	such that $\mathcal{S}_j T$ consists of points in $\supp T$ with tangent cones having at most $j$-dimensional translational invariance. The reader can just regard the Almgren stratification as a geometric measure theory analogue of Whitney stratification.

	Define the $j$-th stratum of the Almgren stratification of $T$ as points in $\supp T$ who have a tangent cone with precisely $j$-dimensional translation invariance. In other words,
	\begin{defn}
		We call $\mathcal{S}_jT\setminus\mathcal{S}_{j-1}T$ the $j$-th stratum of the Almgren stratification, with $S_{-1}T=\emptyset.$
	\end{defn}
	For integers $j\le d-1,$ points in the $j$-th stratum of $T$ necessarily lie in the singular set of $T.$ 
	
	We also need a category of manifolds that serve as the underlying set of singular sets.
	\begin{defn}\label{defnnk}
		We say an $n$-dimensional compact connected manifold $N$ is Nash-Kuiper, if $N$ can be smoothly embedded as a hypersurface into $\R^{n+1}.$
	\end{defn}
	\begin{assump}\label{assumpmff}Assume the following
		\begin{itemize}
			\item	The symbols $d,c$ denote integers with $d\ge2,c\ge 3,$ and $\mff$ is a finite set of disjoint Nash-Kuiper manifolds of distinct dimensions.
			\item For Theorem \ref{thmi} assume that
			\begin{align*}
				d-c+1\le\mathbf{min}_{N\in \mff}\dim N\le\mathbf{max}_{N\in \mff}\dim N\le d-2,	
			\end{align*}
			\item For each element $N\in \mff$, prescribe a closed subset $K_N\s N.$
		\end{itemize}
	\end{assump}
	For instance, with $d=4,c=5$ we can set $\mff=\{S^1,S^2\}$ and $K_{S^1},K_{S^2}$ Cantor sets in $S^1,S^2.$ 
	
	We also need an assumption on the ambient manifold and on the integral homology classes.
	\begin{assump}	\label{assumpbs}
		Assume that
		\begin{itemize}
			\item   $[\Si]$ is a $d$-dimensional integral homology class on  a compact $(d+c)$-dimensional smooth manifold $M$,
			\item $[\Si]$  can be represented by an embedded connected smooth submanifold $\Si$.
		\end{itemize}  
	\end{assump}
	We want to emphasize that the last bullet is not restrictive, and such homology classes $[\Si]$ can be found on almost any manifold. A well-known corollary Thom's classical work \cite{RT} is as follows, e.g., \cite[Lemma 2.1.2]{ZLh},\begin{fact}\label{fctthom}If $M^{d+c}$ is closed and orientable with $c\ge 2$, then for any integral homology class $[\Pi]\in H_d(M,\Z)$ that is not torsion, there exists an infinite subset $I\s N$ such that $[\Si]=i[\Pi]$ with $i\in I$ satisfy Assumption \ref{assumpbs}. 
	\end{fact}By $[\Si]$ not being a torsion, we mean that there is no non-zero integer $l$ such that $l[\Si]=0.$ If the $d$-th Betti number $b_d$ of $M$ is not zero, then there always exists an integral homology class that is not torsion. A corollary of Fact \ref{fctthom} is that if $M$ is closed and orientable, then $H_d(M,\Q)$ can be rationally generated by classes which satisfy Assumption \ref{assumpbs}.  
	
	Under the above assumptions, we can construct area-minimizing integral currents whose singular sets are the disjoint union $\cup_{N\in \mff}K_N$ and each $K_N$ is precisely the $(\dim N)$-th stratum in the Almgren stratification. In other words, we can prescribe almost arbitrarily the structure of $T.$ 
	\begin{thm}\label{thmi}Under Assumption \ref{assumpmff} and \ref{assumpbs}, there exists a smooth metric $g$ on $M$, and an area-minimizing representative $T$ of $[\Si]$, such that
		\begin{enumerate}
			\item $T$ is the image of a smooth immersion of a connected $d$-dimensional manifold into $M$. 
			\item The singular set of $T$ is the disjoint union $$\bigcup_{N\in\mff}K_N,$$   and the regular set of $T$ is connected.
			\item For all $N\in \mff $, the $(\dim N)$-th stratum of the Almgren stratification of the singular set of $T$ is precisely $K_N.$
			\item A $d$-dimensional smooth calibration form $\phi$ on $M$ calibrates $T$ provided $[\Si]$ is not a torsion class.
		\end{enumerate} 
	\end{thm}
	Recall that a set is countably $j$-rectifiable if it can be covered Hausdorff $j$-dimensional a.e. by countable unions of Lipschitz images of subsets of $\R^j.$ Thus,  Theorem \ref{thmi} above shows the rectifiability of singular sets of area-minimizing integral currents established in \cite{DMS1,DMS2,DMS3,BK1,BK2,BK3} is sharp and cannot be improved much.
	
	Let us first recall what we proved in \cite{ZLa} and then we will give a sketch of the proof.
	\subsection{Brief Recall of Part One}\label{secbf}
	To review what we did in \cite{ZLa}, we first need to formally introduce the notion of calibrations, which is almost the only method to prove area-minimizing \cite{HL} known today. Recall that the comass of a  $d$-dimensional differential form  $\phi$ is the maximum of $\phi$ evaluated on unit simple $d$-vectors in the tangent space to $M$ among all  points \cite[Section 1.8]{HF}. In other words
	\begin{align}\label{eqcms}
		\cms_g\phi=\max_{q\in M}\max_{\substack{P\s T_qM\\ \dim P=d}}\phi(P).
	\end{align}
	Here we use the symbol $P$ to denote both a $d$-dimensional plane $P$ and the unit simple $d$-vector representing $P.$
	\begin{defn}(Definition of calibrations)\label{defncal}
		\begin{itemize}
			\item 	We say a closed smooth $d$-form $\phi$ on a (possibly open) ambient manifold is a calibration if its comass is at most $1.$ 
			\item 
			We say a $d$-dimensional integral current $T$ is calibrated by $\phi,$ if $d$-dimensional Hausdorff measure almost everywhere $\phi$ restricted to the tangent space of $T$ equals the volume form of $T.$
		\end{itemize}
	\end{defn}
	The fundamental theorem of calibrated geometry \cite[Theorem 4.2]{HL} gives	\begin{lem}\label{fcal}
		If a $d$-dimensional integral current $T$ is calibrated by a $d$-dimensional calibration form $\phi,$ then $T$ is area-minimizing.
	\end{lem}
	
	Now we are ready to give a brief recall of \cite{ZLa}. For the standard $n$-dimensional complex Euclidean space $\C^n,$ we will always use a holomorphic coordinate system $(z_1,\cd,z_n)$ with the corresponding real coordinate system $(x_1,\cd,x_n,y_1,\cd,y_n)$ defined by 
	\begin{align}\label{cscz}
		z_1=x_1+iy_1,\cd,z_n=x_n+iy_n.
	\end{align} The standard flat metric on $\C^n$ is defined so that the collection of all real coordinate vector fields,  \begin{align}\label{bsxy}
		\{\pd_{x_1},\cd,\pd_{x_n},\pd_{y_1},\cd,\pd_{y_n}\}
	\end{align}
	form an orthonormal basis.
	
	The symbol $\C^n/\Z^{2n}$ will be reserved for the tori obtained from the quotient of $\C^n$ by the lattice spanned by (\ref{bsxy}), i.e.,
	\begin{align*}
		\Z^{2n}=\textnormal{span}_{\Z}\{\pd_{x_1},\cd,\pd_{x_n},\pd_{y_1},\cd,\pd_{y_n}\}.
	\end{align*}
	The natural orthogonal direct sum splitting of 
	\begin{align*}
		\C^n=\textnormal{span}_{\R}\{\pd_{x_1},\cd,\pd_{x_n}\}\oplus\textnormal{span}_{\R}\{\pd_{y_1},\cd,\pd_{y_n}\},
	\end{align*}
	gives a natural splitting of $\C^n$ as a Riemannian product
	\begin{align}\label{cspl}
		\C^n=\cu{x_1\cd x_n}\times \cu{y_1\cd y_n},
	\end{align}
	where $\cu{x_1\cd x_n}$ denotes the $n$-dimensional oriented plane spanned by $\pd_{x_1},\cd,\pd_{x_n}$ and $\cu{y_1\cd y_n}$ denotes the $n$-dimensional orinted plane spanned by $\pd_{y_1},\cd,\pd_{y_n}.$
	
	The product (\ref{cspl}) induces a natural Riemannian product splitting of
	\begin{align*}
		\C^n/\Z^{2n}=X\times Y,
	\end{align*}
	where $X$ is the quotient of $\cu{x_1\cd x_n}$ by $\Z^{2n}$ and $Y$ is the quotient of $\cu{y_1\cd y_n}$ by $\Z^{2n}.$
	
	The symbol $S^1$ will always mean the standard unit circle in $\R^2,$ regarded as a manifold. We will also frequently identify $S^1$ as the interval
	$
	[-{\pi}{},{\pi}{}],
	$ with end points identified and we will reserve label $t$ to denote points in $S^1.$	
	
	Take a compact orientable manifold $N$ which will serve as the base singular set in the coming Lemma \ref{lemr}. The natural splitting of $\C^{d-\dim N}$ as a Riemannian  product
	\begin{align*}
		\C^{d-\dim N}=\cu{x_1\cd x_{d-\dim N}}\times \cu{y_1\cd y_{d-\dim N}},
	\end{align*} 
	induces a natural Riemannian product splitting on the tori
	\begin{align*}
		\cz=X\times Y.
	\end{align*}
	Define an ambient manifold $M_N$ to be
	\begin{align}
		M_N=\C^{d-\dim N}/\Z^{2(d-\dim N)}\times N\times S^1=X\times Y\times N\times S^1,
	\end{align}
	We will reserve the labels $x,y,p,t$ for points in $X,Y,N,S^1,$ respectively. 
	
	The projections $\pi_{X},\pi_Y,\pi_{N\times S^1}$ will denote the canonical projections of $M_N$ onto the respective factors.
	\begin{defn}(Definition 3.1 of \cite{ZLa})
		For any smooth function $$f:N\to (-\pi,\pi),$$ 
		define $N_f$ to be the oriented image of $N$ under the map $\id\times f:N\to N\times S^1,$ i.e.,
		\begin{align*}
			N_f=\{(p,f(p))|p\in N,f(p)\in(-\pi,\pi)\s S^1\}.
		\end{align*}
	\end{defn}
	For example, regarding $0$ as a constant function on $N,$ we get $$N_0=N\times\{0\}\s N\times S^1.$$ From now on, we will identify $N_0$ as a canonically embedded image of $N$ inside $N\times S^1.$
	The key technical result of \cite{ZLa} is the following lemma.
	\begin{lem}\label{lemzla}(Lemma 3.3 of \cite{ZLa})
		\label{lemr}
		Equip the ambient manifold
		\begin{align}\label{defnmn}
			M_N=\C^{d-\dim N}/\Z^{2(d-\dim N)}\times N\times S^1,
		\end{align}with the product Riemannian metric, denoted by $h.$ Then there there exist
		\begin{enumerate}
			\item  Two $d$-dimensional embedded oriented submanifolds 
			\begin{align*}
				X\times N_0, Y\times N_f,
			\end{align*} 
			\item Two smooth retractions 
			\begin{align*}
				\pi_{X\times N_0},\pi_{Y\times N_f},
			\end{align*}of $M_N$ 
			onto  $X\times N_0, Y\times N_f$, respectively.
		\end{enumerate}
		such that
		\begin{enumerate}
			\item In a smooth metric $g_N$ possibly different from $h,$ the smooth closed form
			\begin{align}\label{eqpn1}
				\phi_N=\pxn\du\dvx+\pyn\du\dvy,
			\end{align}is a calibration form that calibrates the integral current $$\Si_{N_f}=X\times N_0+Y\times N_f.$$
			\item The singular set of $\Si_{N_f}$ is $K_N$ in Assumption \ref{assumpmff} and is the $(\dim N)$-th stratum in the Almgren stratification of $\Si_{N_f}.$
		\end{enumerate}
		Here $\dvx,\dvy$ are the volume forms of $\xn,\yn$, respectively, in the $g_N$ metric.
	\end{lem}
	Most of the results in \cite{ZLa} follow from Lemma \ref{lemr}. 
	\subsection{Sketch of Proof for Theorem \ref{thmi} and Theorem \ref{thmcs}}\label{secsk}In Lemma \ref{lemr}, the regular set of $\Si_{N_f}$ has two connected components and the ambient manifold $M_N$ has a very special topological structure. To overcome these two caveats, we need to introduce the framework of gluing calibrations developed by Yongsheng Zhang in his Ph.D. thesis \cite{YZa,YZj,YZt}.
	
	Zhang's gluing of calibrations is a bit like Nash-Moser iteration. It is a framework of ideas instead of a simple plug-and-play algorithm. The central idea is to split the metric into directions adapted to the problem so that by making the metric large in irrelevant directions \cite[Lemma 2.14]{HLf} we can keep the calibrations form calibrations when doing gluing constructions. 
	
	The reader can try evaluating the complicated expression (\ref{eqcms}) with their favorite differential forms and will quickly realize that calculating the comass exactly is mission impossible. Thus, Frank Morgan has famously commented \cite[p. 343]{FMct}
	\begin{quote}
		\emph{Finding a calibration remains an art, not a science.}
	\end{quote}
	In view of the above quote, the reader can imagine that implementing Zhang's gluing of calibrations in different cases is necessarily a distinct technical task, and thus will form the core of our manuscript.
	
	Now back to our idea of proof, to resolve the two connected components of the regular set of $\Si_{N_f}$, we need use a submanifold connected sum to connect them to construct a new cycle $\ov{\Si_{N_f}}$, and use Zhang's gluing of calibrations to prove that calibrations can be preserved under connected sum. 
	
	To prove Theorem \ref{thmi}, let us first recall Assumption \ref{assumpbs}. The topological representative $\Si$ has no singular set to start with. Thus, to achieve our goal, we must add singular sets manually to alter the topological representative $\Si$ while achieving the following features:
	\begin{enumerate}
		\item  The added singular sets give an altered topological representative that stays in the same homology class $[\Si]$.\label{ft0}
		\item The added singular sets should have satisfied the conclusion of Theorem \ref{thmi}.\label{ft1}
		\item We can find a smooth Riemannian metric $g$, in which the altered topological representative is area-minimizing.\label{ft2}	
	\end{enumerate}
	Features (\ref{ft0}) and (\ref{ft1}) are only about the altered topological representative, while in Feature (\ref{ft2}) the altered topological representative is area-minimizing in our newly found metric $g.$
	
	To achieve Feature (\ref{ft0}) and (\ref{ft1}), for $N\in\mathcal{F}$ in Assumption \ref{assumpmff}, we will prove that 
	\begin{fact}\label{fctnk}The following manifold $M_N$ admits smooth embedding with trivial normal bundles into the standard $(d+c)$-dimensional unit ball of $\R^{d+c}:$
		\begin{itemize}
			\item	$
			M_N=\C^{d-\dim N}/\Z^{2(d-\dim N)}\times N\times [-3,3],
			$ if $\dim N=d-2$ and $c=3$,
			\item $M_N$ if $\dim N\le d-3$ or $c\ge 4.$
	\end{itemize}\end{fact}We can embed $M_N\times B_1^{c-(d-\dim N+1)}$ as a smooth open set of a standard ball in $M$ near $\Si$. Here $B_1^{m}$ is the $m$-dimensional standard unit ball and we regard $B_1^0$ as a single point. Then a cycle connected sum will make sure that we have an altered representative $$\Si\#_{N\in\mathcal{F}}\ov{\Si_{N_f}},$$with the structure of singular sets as prescribed in Theorem \ref{thmi}.
	
	For Feature (\ref{ft2}), we will develop a version of Zhang's gluing calibrations for boundary connected sums of manifolds. Then we will use Zhang's gluing of calibrations for the connected sum of cycles. As a by-product, we prove Theorem \ref{thmcs}.
	\subsection{Overview of the paper}
	Our paper will be structured as follows. 
	In Section \ref{bsdefn}, we will give basic definitions and collect several elementary facts.
	
	In Section \ref{seccs}, we will construct the connected sum of calibrated currents inside one ambient manifold a la Zhang.
	
	In Section \ref{seccs2}, we will construct calibration forms for connected sums of ambient manifolds a la Zhang and finish the proof of Theorem \ref{thmcs}.
	
	In Section \ref{secpf}, we will prove Theorem \ref{thmi}.
	\section*{Acknowledgements}
	The author expresses profound gratitude to Professor Hubert Bray, who has generously shared invaluable advice on nearly every aspect of this world, and whose encouragement ultimately inspired the author to pursue Almgren’s conjecture. Without Professor Bray’s extraordinary guidance, kindness, and unwavering support, the author would not have been able to pursue a mathematical career. On the occasion of his 55th birthday, the author is delighted to extend warmest wishes to Professor Bray.
	
	The author is also deeply indebted to Professor Leon Simon for his exceptional generosity, insight, and kindness over many years. Professor Simon’s work \cite{LSfr} has been a lasting and fundamental source of inspiration for the author.
	
	The author cannot thank his Ph.D. advisor, Professor Camillo De Lellis, enough for his steadfast support throughout the years, who introduced this problem to the author and offered countless insightful suggestions.	
	
	Last but not least, the author would like to thank the referees of \cite{ZLa}, who have recommended writing up the results in this manuscript.
	\section{Basic definitions and preliminaries}\label{bsdefn}
	We will give some basic definitions in this section. The contents are more or less standard, except for Sections \ref{secnb} and \ref{secnk}, and the experienced reader can skip this section on the first reading.
	\subsection{Manifolds}\label{bdmn}
	In general, we will use the symbol $d$ as the dimension of the integral currents and the symbol $c$ as the codimension of the current with respect to the ambient manifold. The symbol $M$ will be reserved for a $(d+c)$-dimensional closed ambient manifold and the symbol $\Si$ will be reserved for a $d$-dimensional integral current of codimension $c$ inside $M.$ 
	
	We will reserve the boldface symbol $\mathbf{T}$ to mean the tangent space to a manifold or a submanifold.
	
	We will often speak of smooth neighborhoods or smooth open sets, by which we mean an open set with a smooth boundary.
	
	We will also use transversality \cite{MH} frequently. By this we mean that whenever we have a  pair of smooth submanifolds of dimension $d,d'$ on an ambient manifold of dimension $(d+c),$ respectively, we can arrange by an arbitrarily small smooth perturbation so that the pair is transverse to each other and thus intersects along a submanifold of dimension $(d'-c)$.	
	
	We also need a definition of orientability along a curve.
	\begin{defn}\label{defnor}If $W,W'$ are two $d$-dimensional orientable submanifolds of the same dimension on $M$ and $\ga$ is a curve from $p\in W$ to $q\in W'$ which meets $W,W'$ transversely and only at $p,q$. We say $\ga$ is orientation reversing, if there exists a coordinate chart $(x_1,\cd,x_{d+c})$ in a neighborhood $U(\ga)$ of $\ga$ such that when restricted  $U$
		\begin{itemize}
			\item $W$ is the oriented plane $(x_1,\cd,x_d,0\cd,0, -\frac{1}{2})$, 
			\item $W'$ is the plane $(x_1,\cd,x_d,0\cd,0, \frac{1}{2})$ with reverse orientation,
			\item $\ga$ is the line segment from $(0,\cd,0,-\frac{1}{2})$ and $(0,\cd,0,\frac{1}{2}).$
		\end{itemize} If either $W$ or $W'$ are unorientable, we regard every curve from $W$ to $W'$ as orientation reversing.
	\end{defn}For a smooth curve $\ga$ from between two points, its normal bundle is always trivial, so the above bullet conditions always hold if we drop the orientability assumptions. The orientability assumption ensures that we can thicken $\ga$ to a neck as a connected sum.
	\subsection{Riemannian geometry}\label{basrie}
	We will use $\textnormal{dist}(p,K)$ to denote the Riemannian distance between a point $p$ and a closed set $K.$ 
	
	We will reserve the symbol $\no{\cdot}$ to denote the Riemannian length of vectors, forms, etc. Whenever we use the symbol $|\cdot|,$ it is understood that we are either taking the absolute value or taking the square root of the sum of squares of components in coordinates, thus different from $\no{\cdot}$ on general manifolds.
	\subsection{Planes, vectors and forms}
	For a $d$-dimensional oriented plane $P$ at a point in the tangent space of the ambient manifold, we will also use the symbol $P$ to denote the unique simple unit $d$-vector generating $P$ and will use $P^\ast$ to denote the $d$-form Riemannian dual to $P.$ 
	
	Conversely, when we write a unit simple $d$-vector $v_1\w\cd\w v_d$, it can also mean the $d$-dimensional oriented plane spanned by $v_1,\cd,v_d.$
	\subsection{Several facts about comass}
	With the definition of comass in Section \ref{secbf} in mind, we record some facts about comass on \textbf{vector spaces}. We will frequently apply the below fact \textbf{pointwise} in tangent spaces of ambient manifolds when constructing calibrations and calculating comass.
	\begin{fact}\label{cmsvec}
		Let $g,h$ be positive definite quadratic forms on a finite-dimensional vector space, with $\psi$ a constant differential form, i.e., a form with constant coefficients everywhere with respect to a fixed dual basis.\begin{enumerate}
			\item $\cms_h\psi\le\cms_g\psi$ if $h\ge g$ as quadratic forms.\label{cms1}
			\item $\cms_{\lam^2 h}\psi=|\lam|^{-\dim\psi}\cms_h\psi$.\label{cms2}
			\item If $\phi$ is a simple form, then $\cms_g\phi=\no{\phi}_g$, i.e., its comass equal to its Riemannian length.\label{cms0}
		\end{enumerate}
	\end{fact}
	Here the parameter $\lam\in \R$, the symbol
	$\cms_g\psi$ means the comass of $\psi$ with respect to $g$. The first two bullets are proved in \cite[Lemma 2.1.9, 2.1.20]{YZt}. The last bullet is a direct calculation using Cauchy-Binet.
	\subsection{The normal bundle calibration}\label{secnb}
	Let $W$ be a closed $w$-dimensional smooth submanifold of a closed Riemannian manifold $(V,h).$
	
	Suppose a compact set $U(W)$ containing $W$ lies in the bijective image of the normal bundle exponential map $\exp^\perp_W$ of $W.$
	\begin{lem}\label{lemznb}
		The closed $w$-form $$\Pi_W\du\dvol_W^h$$ calibrates $W$ in $U(W)$ with the smooth metric $$h'=\no{\Pi_W\du\dvol_W^h}_h^{\frac{2}{w}}h.$$
	\end{lem}
	Here $\Pi_W$ is the nearest distance projection onto $W,$ $\dvol_W^h$ is the volume form of $W$ in $h$ and $\no{\cdot}_h$ is the Riemannian length in $h.$
	\begin{proof}
		\cite[Lemma 2.4]{ZLa}\end{proof}
	\subsection{Smooth functions with controlled zero set}
	Let $N$ be a smooth orientable Riemannian manifold. Let $K$ be a compact closed subset of $N.$ We need the following lemma.
	\begin{lem}\label{lemzs}
		There exists a smooth function $f$ on $N$ such that the zero set of $f$ is $K, $ i.e., $$f\m(0)=K,$$ and $f$ vanishes to infinite order at $K.$
	\end{lem}
	\begin{proof}This is a well-known fact. We only give a sketch here.
		By the Whitney embedding theorem in \cite{CW} and the fact that a smooth function restricts to a smooth function on submanifolds, we reduce to the case of finding a smooth function vanishing with zero set $K$ and infinite order vanishing at $K$ in Euclidean space. By \cite[3.1.13]{HF}, there exists a good cover of the complement of $K$ with a uniform bound on the number of intersections. Then for each ball in the covering, we can take a bump function supported on the ball and non-vanishing in the interior of the ball. Now take a sum of these bump functions with small constants in front of each term, so that the sum of $C^k$ norm is absolutely convergent for all $k$. Thus the infinite sum is smooth, and by construction has zero set precisely $K$. After restricting to an embedding of $N, $ we can adjust the $C^l$ norm of the infinite sum by multiplying with small constants.
	\end{proof} 
	\subsection{Definition of area-minimizing}
	The right category to discuss area-minimizing representatives is the category of integral currents and integral currents modulo $v$, with integer $v$ at least $2$.  We will use Federer's definitive monograph \cite{HF} and Simon's classical lecture notes \cite{LS} as basic references. 
	
	For our purposes, the reader can just regard integral currents as integer coefficient simplicial chains in algebraic topology. In reality, integral currents can be defined \cite[4.1.24]{HF} as the closure of the set of simplicial chains under mass topology applied both to the chains and their boundaries. Similarly,  the reader can just regard integral currents mod $v$ as limits of $\Z/v\Z$-coefficient simplicial chains in algebraic topology. Note that integral currents and integral currents mod $v$ are allowed to have finite area boundaries.
	
	We reserve the symbol $\supp$ for the underlying rectifiable set of an integral current.
	
	We say an integral current $T$ is area-minimizing, if $T$ has the least area among all integral currents homologous to $T$:
	\begin{defn}\label{defnam}
		An $d$-dimensional integral current $T$ is area-minimizing if 
		\begin{align*}
			\ms(T)\le \ms(T+\pd V),
		\end{align*}for all  $d+1$-dimensional integral currents $V$. 
	\end{defn}
	Here $\ms$ is the mass of the current, which in our context, is just the area of the underlying set with multiplicity included. For instance, $\ms(\pm 2 S^1)=2\ms(S^1)=4\pi$ for the unit circle $S^1$ in $\R^2.$
	
	The primary tool to prove area-minimizing is calibrations, which the reader can recall from Section \ref{secbf} and the main reference is \cite{HL}.
	
	\subsection{Definition of singular sets}
	\begin{defn}\label{defnsm}
		We say an integral current $T$ is smooth at a point $p$ in the support of $T$ if there exists an open set $U$ containing $p$ on $M$, such that  $T$ restricted to $U$ equals an integer multiple of a smooth submanifold $N.$ The definitions of regular sets and singular sets of $T$ are as follows:
		\begin{itemize}
			\item The singular set of $T,$ $\sing T,$ is defined as set of points in the support of $T$ where $T$ is not smooth.		
			\item 
			The regular set of $T,$ $\operatorname{Reg}T$, is defined as the set of the points in the support of $T$ where $T$ is smooth. 
		\end{itemize}
	\end{defn} 
	Here support means the underlying set of an integral current. From now on, we will also reserve the symbol $\supp T$ to mean the underlying set of $T.$ 
	
	For example, the figure $8$ has a singular point at its self-intersection, while the figure $0$ counted with multiplicity $-2$ is smooth. We want to emphasize that in the above definition, any point on the boundary of a current is a singular point. For instance, the regular set of the figure $3$ consists of two open arcs.
	\subsection{Area-minimizing integral currents in Riemannian products and disjoint unions}\label{basdis}
	With the definition of area-minimizing in the previous subsection, we will collect several facts about area-minimizing integral currents.
	\begin{lem}\label{lemdis}
		Suppose the $d$-dimensional integral currents $T_1,\cd,T_n$ are area-minimizing (in integral or mod $v$ homology) inside compact Riemannian manifolds $M_1,\cd,M_n$, respectively. Then the sum of currents
		\begin{align*}
			T_1+\cd+T_n
		\end{align*}
		is area-minimizing (integral or mod $v$, respectively) in the disjoint union manifold
		\begin{align*}
			M_1\cup\cd \cup M_n.
		\end{align*}
		Furthermore, if $T_1,\cd,T_n$ are calibrated by $\phi_1,\cd,\phi_n$ on $M_1,\cd,M_n$, respectively, then $T_1+\cd+T_n$ is calibrated by the unique closed form whose restriction to $M_1,\cd,M_n$ is $\phi_1,\cd,\phi_n$, respectively. 
	\end{lem}
	\begin{proof}
		Lemma 2.9 of \cite{ZLa}\end{proof}
	\begin{lem}\label{lemprod}
		Suppose the $d$-dimensional integral current $T$ is area-minimizing (in integral or mod $v$ homology) on a compact  Riemannian manifold $V$. Let $W$ be another  compact Riemannian manifold with $p$ a point in $W$. Define a map $i_p:V\to V\times W,$ by
		$		i_p(q)=(q,p).
		$	Then the pushforward current $
		(i_p)\pf(T)$
		is area-minimizing (in integral or mod $v$ homology) on $V\times W$. Furthermore, if $T$ is calibrated by a smooth form $\phi$ on $V,$ then $(i_p)\pf T$ is calibrated by a smooth form $		\pi_V\du\phi,$ on $V\times W,$ where $\pi_V$ is the canonical projection of $V\times W$ onto the $V$ factor.
	\end{lem}
	\begin{proof}
		Lemma 2.10 of \cite{ZLa}.\end{proof}
	\subsection{Zhang's constructions of extending calibrations}We need a way to extend calibrations from a smooth open set of the ambient manifold onto the entire manifold.
	\begin{lem}\label{lemzhang}
		Assume that,
		\begin{itemize}
			\item 	$T$ is a $d$-dimensional integral current on a $(d+c)$-dimensional compact closed Riemannian $(M,g)$, and $U(T)$ is a smooth open set containing $\supp T$,
			\item $\phi$ is a calibration form on $U(T)$ with respect to a metric $h$ on $M$.
			\item the $d$-th homology group of $U(T)$ is generated by the homology class of $[T]$, i.e., $H_d(U(T),\Z)=\Z[T],$
		\end{itemize}	 then there exists a smooth Riemannian metric $g$ on $M$, such that,
		\begin{itemize}
			\item $T$ is area-minimizing on $(M,g),$
			\item if $T$ represents a non-zero $\R$-homology class, then there exists a smooth calibration form $\ov{\phi}$ calibrating $T$ and coinciding with $\phi$ in a neighborhood of $\supp W$. 
		\end{itemize}
	\end{lem}
	\begin{proof}
		The result is due to Yongsheng Zhang's Ph.D. thesis \cite{YZt} and published in \cite{YZj,YZa}. 
		
		Let us first deal with the case of $[T]=0\in H_d(M,\R).$ We claim that we can find a smooth metric $h$, such that
		\begin{itemize}
			\item any stationary varifold on $(M,g)$ whose support is not contained in $U(T)$ has area larger than $2 \ms_h^{\Z}(N),$
			\item $\phi$ calibrates $T$ uniquely on $(U(T),g).$
		\end{itemize}
		The above two bullets are from \cite[Section 3,4]{YZa}. Roughly speaking, we achieve this by making the metric large away from $\supp T.$ The condition of $H_d(U(T),\Z)=\Z[T]$ then ensures that $T$ is area-minimizing in $(M,g)$.
		A detailed write-up of the entire argument is in \cite[Section 6.4.7-6.4.8]{ZLns}. Though \cite[Section 6.4.7-6.4.8]{ZLns} is stated for a special class of $T,$ the only condition we used therein is the existence of a neighborhood of $\supp T $, whose $d$-th homology is generated by $T.$
		
		The case of $[T]$ representing a non-trivial real homology class needs a different proof. The proof is \cite[Section 3.3]{YZa}. Though \cite[Section 3.3]{YZa} is stated for $T$ smooth and $\phi$ a special form, the proof also only uses the existence of a neighborhood around $\supp T$ whose $d$-th homology group is generated by $T$. Roughly speaking, taking any closed form $\phi'$ on $M$ such that $\phi'(T)=\ms_h(T)$, by $H_d(U(T),\Z)=\Z[T],$ we know that $\phi'|_{U(T)}$ is homologous to $\phi.$ Then obtain a closed smooth form $\ov{\phi}$ that patches $\phi$ to $\phi'$ near $\pd U(T)$ and alter the metric to keep the glued form a calibration.
	\end{proof}\subsection{Proof of Fact \ref{fctnk}}\label{secnk}
	Let us collect several facts about Nash-Kuiper manifolds, from which Fact \ref{fctnk} follows immediately.
	\begin{fact}\label{fctnknb}
		Every $n$-dimensional Nash-Kuiper manifold $N$ is orientable and admits an embedding with trivial normal bundle into the unit ball $B^{n+k}_1$ of $(n+k)$-dimensional Euclidean space with $k\in \Z_{\ge 1}.$
	\end{fact}
	\begin{proof}
		By	\cite{HSoh}, every smooth hypersurface of Euclidean space is orientable, thus having a trivial normal bundle. Composing an embedding of $N$ into $B^{n+1}$ with the standard embedding $\R^{n+1}\to\R^{n+1}\times\{0\}^{k-1}$ for $k\ge 2.$ We are done.
	\end{proof}
	\begin{fact}
		Products of Nash-Kuiper manifolds are Nash-Kuiper.
	\end{fact}
	\begin{proof}\label{fctnkpd}
		Suppose we have two Nash-Kuiper manifolds $N_1^{n_1},N_2^{n_2}$. By Fact \ref{fctnknb}, let $i_1,i_2$ be embeddings with trivial normal bundle of $N_1^{n_1},N_2^{n_2}$ into the unit balls $B_1^{n_1+n_2+1},B_1^{n_2+1}$, respectively. Without loss of generality, we can assume that the image of $i_1,i_2$ avoids the origin.
		
		Since $i_1$ has trivial normal bundle, a closed neighborhood $U$ of $i_1(N_1)$ is diffeomorphic to $N_1\times B_1^{n_2+1}$, i.e.,
		\begin{align*}
			\Ga(N_1\times B_1^{n+2_1})=U,
		\end{align*}for some diffeomorphism $\Ga.$ Then $\Ga(N_1\times i_2(N))$ is an embedding of $N_1\times N_2$ into $\R^{n_1+n_2+1}.$ We are done.
	\end{proof}
	Both bullets of Fact \ref{fctnk} follow from Fact \ref{fctnknb} and Fact \ref{fctnkpd}.
	\section{Calibrated connected sums of multiplicity $1$ integral currents in the same ambient manifold}\label{seccs}
	In this section, we will prove a powerful theorem about the connected sum of calibrated currents by Yongsheng Zhang (\cite{YZa,YZj}), which combined with results in next section will be used to deduce Theorem \ref{thmi} from Lemma \ref{lemzla}. Roughly speaking, Zhang conceived a very general framework of gluing calibrations. We apply his vision to our argument here.
	
	Before stating the theorem, we need to define a special class of integral currents called multiplicity $1$ currents.
	\begin{defn}\label{defnmo}
		We say an integral current $T$ is of multiplicity $1,$ if $T$ restricted to a neighborhood of each smooth point $p$ equals a smooth submanifold $N$.
	\end{defn}
	For example, a figure $88$ in $\R^2$ is a multiplicity $1$ integral current.
	\begin{assump}\label{assumpcs}
		Assume that \begin{enumerate}
			\item $T$ and $V$ are two $d$-dimensional multiplicity $1$ integral currents of dimension $d$ and codimension $c$ both at least $1$ inside a connected not necessarily closed manifold $M$. 
			\item $\reg (T+V)$ contains at least $2$ points $p$ and $q$, such that $p\in\reg T\setminus\supp V$ and $q\in\reg V\setminus\supp T,$ respectively,\label{defnregtv}
			\item If $c=1$, there exists a smooth curve $\ga$ from $p$ to $q$, so that $\ga$ intersects $T+V$ only at $p,q$, the intersections at $p,q$ are transverse and $\Ga$ is orientation reversing.\label{asscsc=1}
			\item $T,V$ are calibrated by a smooth $d$-form $\phi_\sing$ in a smooth metric $h_\sing.$\label{defnpsing}
		\end{enumerate}
	\end{assump}	
	For $c=1$, the complicated assumptions in bullet (\ref{asscsc=1}) is necessary to ensure the existence of connected sums of chains on the topological level. Orientation reversing is defined in Definition \ref{defnor}.
	
	Our goal is to construct a connected sum $T\# V$. 
	\begin{lem}\label{lemcs}Under Assumption \ref{assumpcs},
		there is a multiplicity $1$ integral current $T\# V$ on $M$, called the connected sum of $T$ and $V,$ so that 
		\begin{itemize}
			\item $T\# V$ is homologous to $T+V$.
			\item $T\#V$ and $T+V$ coincides outside of a smooth open set $U,$ and the restrictions of $T\#V$ and $T+V$ to $\ov{U}$ are both smooth.
			\item $T\# V$ has connected regular set provided the regular set of $T+V$ has only two connected components and $d\ge 2$.
			\item We have $$\sing (T\# V)=\sing (T+V)\s \ov{U}\cp.$$
			\item If $T,V$ are images of immersions of connected manifolds $F,G$, then $T\#V$ is the image of an immersion from the manifold $F\# G$.
			\item $T\# V$ is calibrated by a smooth form $\phi$ in a smooth metric $h.$
		\end{itemize}
	\end{lem}
	For instance, the connected sum of the connected sum of two distinct irreducible holomorphic subvarieties of the same dimension in $\mathbb{CP}^n$ is calibrated by a smooth form in a modified metric.
	
	If we drop the condition of calibrations, in case $T$ and $V$ are disjoint submanifolds, the construction is classical and can be found in, e.g., \cite{CL}. 	In our case, $T$ and $V$ are general calibrated integral currents that  may intersect each other and may have exotic singular sets. However, the idea is similar. We need to remove a $d$-dimensional ball from $T$ and $V,$ respectively, then use a tube diffeomorphic to $S^{d-1}\times[-\frac{1}{2},\frac{1}{2}]$ to connect $T$ to $V.$ Finally, in a neighborhood around the tube $S^{d-1}\times[-\frac{1}{2},\frac{1}{2}]$ we glue the calibrations and metrics in Lemma \ref{lemznb} to the calibrations and metrics in Assumption \ref{assumpcs}.
	
	This section will be devoted to the proof of Lemma \ref{lemcs}.
	
	Roughly speaking, the proof consists of four steps.
	\begin{enumerate}
		\item We construct $U$ as a coordinate chart, in which $T$ and $V$ restrict to be two disjoint $d$-dimensional balls.
		\item We remove replace $T+V$ restricted to $U$ with a tube $S^{d-1}\times[-\frac{1}{2},\frac{1}{2}]$ and prove everything except for the last bullt.
		\item We apply Zhang's vision in \cite{YZa,YZj} to make $T\# V$ calibrated.
	\end{enumerate}
	Let us start with the first step.
	\subsection{Construction of coordinate chart $U$}\label{seccht}
	Take $p\in \reg T\setminus \supp V$ and $q\in \reg V\setminus\supp T,$ and a smooth curve $\ga$ from $p$ to $q$. If codimension $c=1$, take $\ga$ as in bullet (\ref{asscsc=1}) in Assumption \ref{assumpcs}. The idea is to construct $U$ as a coordinate chart along $\ga$.
	
	Let us first consider the case of the codimensions of $T,V$ being $c\ge2$. We can use transversality \cite{MH}, i.e., taking a generic perturbation of $\ga,$ to achieve that\begin{align*}
		\ga\cap\supp T=p,\text{ }\ga\cap\supp V=q,
	\end{align*}and there exist two coordinate charts $$U_p,U_q$$ centered around $p,q$ respectively, with disjoint closures, i.e., $\ov{U_p}\cap\ov{U_q}=\es.$ Furthermore, we can require that when restricted to $U_p$ with coordinate labels $$(p_1,\cd,p_{d+c})$$and
	\begin{align*}
		\sum_{j=1}^{d+c} p_j^2\le 36,
	\end{align*} we have\begin{align}\label{crp}
		T=\cu{p_1\cd p_d},\ga=\cu{(p_{d+c})_{\ge0}},
	\end{align} and restricted to $U_q$ with coordinate labels $$(q_1,\cd,q_{d+c})$$and\begin{align*}
		\sum_{j=1}^{d+c} q_j^2\le 36,
	\end{align*}  we have\begin{align}\label{crq}
		V=-\cu{q_1\cd q_d},\ga=\cu{(q_{d+c})_{\le0}}.
	\end{align}
	Here $\cu{(p_{d+c})_{\ge0}}$ means the nonnegative $p_{d+c}$-axis and $\cu{(q_{d+c})_{\le0}}$ means the non-positive $q_{d+c}$-axis.
	
	The next step is to make a new coordinate chart $U$ that contains both $U_p,U_q$ and $\ga.$
	
	Equip $U_p,U_q$ with the standard Euclidean metric with respect to the coordinate labels in chart and extend to a smooth ambient metric on $M.$ Then all coordinate vector fields in $U_p,U_q$ are orthonormal parallel frames along $\ga.$ 
	
	Parallel translate $\pd_{p_1},\cd,\pd_{p_{d+c-1}}$ along $\ga$, we arrive at a smooth orthonormal frame $$e_1,\cd,e_{d+c-1}$$ in the normal bundle to $\ga$.
	
	When we restrict the orthonormal frame $$e_1,\cd,e_{d+c-1}$$ to $U_q,$ we obtain a smooth $$\operatorname{O}(d+c-1)$$valued transformation $$\si$$ in the normal bundle of $\ga$, such that
	\begin{align}\label{rotfm}
		\begin{pmatrix}
			e_1\\\vdots \\e_{d+c-1}
		\end{pmatrix}=\si 
		\begin{pmatrix}
			\pd_{q_{d+c-1}}\\\vdots \\\pd_{q_{d+c-1}}
		\end{pmatrix}
	\end{align}
	By codimension $c\ge 2$ in Assumption \ref{assumpcs}, if $\si$ is not  $$\operatorname{SO}(d+c-1)$$valued, then we can reverse the $q_{d+1}$-axis in $U_q$ to make $\si$ into $\operatorname{SO}(d+c-1)$ valued without changing the coordinate representations (\ref{crp}) and (\ref{crq}).
	
	We want to transit smoothly from orthonormal frames $e_1,\cd,e_{d+c-1}$ to $\pd_{q_1},\cd,\pd_{q_{d+c-1}}$ in $U_q$. Since $\operatorname{SO}(d+c-1)$ is a compact connected Lie group, we can find a smooth function $\si':[-1,0]\to \operatorname{SO}(d+c-1),$ such that
	\begin{align*}
		\si'=\begin{cases}
			\si,&\textbf{ on }[-1,-\frac{2}{3}],\\
			\id,&\textbf{ on }[-\frac{1}{3},0].\end{cases}
	\end{align*}

	Define the smooth orthonormal frames 
	\begin{align*}
		e_1',\cd,e_{d+c-1}',
	\end{align*}
	in the normal bundle of $\ga$ in $U_q$ defined by
	\begin{align*}
		\begin{pmatrix}
			e_1'\\\vdots \\e_{d+c-1}'
		\end{pmatrix}=\si' 
		\begin{pmatrix}
			\pd_{q_{d+c-1}}\\\vdots \\\pd_{q_{d+c-1}}
		\end{pmatrix}.
	\end{align*}
	We can obtain a smooth extension of the frames $e_1',\cd,e_{d+c-1}'$  along $\ga$ by setting  $e_1'=e_1,\cd,e_{d+c-1}'=e_{d+c-1}$ on $ U_q\cp\cap \ga.$
	
	Reparameterize $\ga$ as a mapping $[-\frac{1}{2},\frac{1}{2}]\to M,$ with $\ga(-\frac{1}{2})=p,\ga(\frac{1}{2})=q.$ Adopt a Fermi coordinate chart $$U_e$$ along $\ga$ with respect to the orthonormal frames
	\begin{align*}
		e_1',\cd,e_{d+c-1}',
	\end{align*}of the normal bundle of $\ga.$
	Then $U_e$ patches $U_p$ and $U_q$ smoothly together by changing coordinate labels and origins. From $U_e,U_p,U_q$ we get a new coordinate chart $U,$ with coordinate labels $$(x_1,\dots,x_d,y_1,\dots,y_c),$$
	such that
	\begin{align*}
		p=(\overbrace{0,\dots,0}^{d},\overbrace{0,\dots,0}^{c-1},-\frac{1}{2}),q=(\overbrace{0,\dots,0}^{d},\overbrace{0,\dots,0}^{c-1},\frac{1}{2})
	\end{align*}Moreover,  the integral current $T$ restricted to $U$ is parameterized by
	\begin{align*}
		{(x_1,\dots,x_d,0,0,\dots,-\frac{1}{2}),}	\end{align*}with each $x_j\in\R,$and the integral current $V$ restricted to $U$ is parameterized by
	\begin{align}\label{xdminus}
		(-x_1,x_2,\dots,x_d,0,\dots,0,\frac{1}{2}),
	\end{align}with each $x_j\in\R.$ The minus sign is to give opposite orientations in coordinate in order to carry out the connected sum. 
	
	If $c=1$, then the same procedure and conclusion hold due to bullet (\ref{asscsc=1}) of Assumption \ref{assumpcs}.
	\begin{figure}
		\centerline{
			\includegraphics[angle=270,width=0.9\paperwidth]{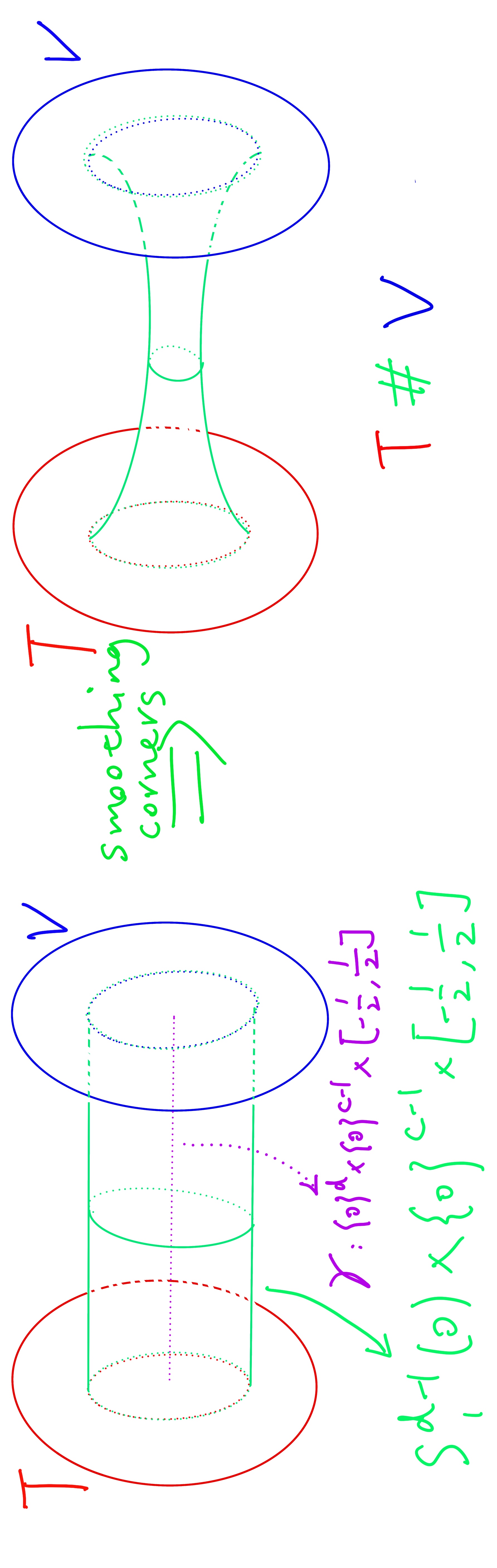}}
		\caption{Replacing balls with a tube}
		\label{figcs}
	\end{figure}
	\subsection{Replacing balls with a tube}
	Now remove the $d$-dimensional ball centered at $p$ $$B_1^d(0)\times\{0\}^{c-1}\times\{-\frac{1}{2}\},$$ from $T$ and remove the $d$-dimensional ball centered at $q$ $$B_1^d(0)\times\{0\}^{c-1}\times\{\frac{1}{2}\},$$ from $V$, and  glue the tube
	\begin{align*}
		S_1^{d-1}(0)\times \{0\}^{c-1}\times [-\frac{1}{2},\frac{1}{2}],
	\end{align*}smoothly onto $$T+V$$ with the two $d$-dimensional balls removed,
	we obtain $$T\# V.$$ Here by gluing smoothly we mean smoothing the corner introduced by the boundary of the tube \cite[Section 2.6]{CW} and we regard $S^0_1$ as two disjoint points.
	
	For the first bullet in Lemma \ref{lemcs}, since
	\begin{align*}
		&	S_1^{d-1}(0)\times \{0\}^{c-1}\times [-\frac{1}{2},\frac{1}{2}]\\=&\pd\bigg(B_1^d(0)\times\{0\}^{c-1}\times [-\frac{1}{2},\frac{1}{2}]\bigg)	+B_1^d(0)\times\{0\}^{c-1}\times\{-\frac{1}{2}\}-B_1^d(0)\times\{0\}^{c-1}\times\{\frac{1}{2}\},
	\end{align*}we deduce that $[T\# V]=[T]+[V]$. The minus sign in $-B_1^d(0)\times\{0\}^{c-1}\times\{\frac{1}{2}\}$ explains the minus sign in parameterization (\ref{xdminus}) of $V$.
	
	The second bullet in Lemma \ref{lemcs} follows by construction of $U$ and $T+V.$ The third bullet in Lemma \ref{lemcs} follows from bullet its assumption of $\reg(T+V)$ having only two connected components and dimension $d\ge 2$. The fourth bullet in Lemma \ref{lemcs} follows because we have only changed $T$ and $V$ inside $U$, where both are regular and have disjoint support. The fifth bullet of Lemma \ref{lemcs} follows directly by construction.
	\subsection{Preparing a neighborhood}
	We are now only left with the last bullet of Lemma \ref{lemcs}, i.e., calibrating $T\# V.$  Before proceeding to the proof, we first need to do some preparational work.
	
	Note that when restricted to $U,$ by construction, $T\# V$ is diffeomorphic to $$S^{d-1}\times[-6,6],$$ and coincides with $T+V$ on 
	\begin{align}\label{sdi}
		\bigg(S^{d-1}\times[-\frac{1}{2},\frac{1}{2}]\bigg)\cp
	\end{align} under this identification. Moreover, $T\# V$ restricted to $U$ has a trivial normal bundle.
	
	Now equip $M$ with the metric $h_\sing$ (Assumption \ref{assumpcs}). Then there exists 
	\begin{align*}
		\mathbf{r}_0>0,
	\end{align*} such that the normal bundle exponential map of $\tvu$ restricted to normals of length at most $\mathbf{r}_0$ is bijective. Here $(T\# V)|_U$ means the restriction of $(T\# V)$ to $U$ and we regard it as an oriented submanifold.
	\begin{fact}\label{fctga}
		Parametrize the normal bundle of $\tvu$ by $S^{d-1}\times[-6,6]\times \R^c.$ Then we have a diffeomorphism $$\Ga,$$ which is the normal bundle exponential map of $\tvu$ parametrized, such that
		\begin{align}
			\Ga(S^{d-1}\times[-6,6]\times\{0\})=&(T\#V)|_U,\\
			\Ga(S^{d-1}\times[-6,6]\times B_{\mfr}^c)=&B_{\mfr}((T\#V)|_U),
		\end{align} for all $\mfr\le \mfr_0.$
	\end{fact}Here $B_{\mfr}^c$ means the standard $d$-dimensional ball of radius $\mfr$ in $\R^c,$ and $B_{\mfr}((T\#V)|_U)$ is the tube of radius $\mfr$ around $(T\#V)|_U.$
	
	From now on, all of our constructions will happen in $B_\mfr(\tvu),$ and we will use the identification of $B_\mfr(\tvu)$ as $S^{d-1}\times [-6,6]\times B_\mfr^c$ via $\Ga.$
	\subsection{Another calibration of $\tvu$}Now we restrict to $\ur$ with a fixed value of \begin{align}
		\mfr\le \mfr_0,\label{mfr0}
	\end{align} to be determined later.
	
	Recall (\ref{defnpsing}) of Assumption \ref{assumpcs}. By construction, in $B_\mfr(\tvu)$, the nearest distance projection \begin{align*}
		\Pi,
	\end{align*}onto $T\# V$ 
	is well defined, and by Lemma \ref{lemznb}, 
	\begin{align*}
		\phi_\reg=\Pi\du\dvol_{T\# V}^{h_\sing}
	\end{align*}
	calibrates $T\# V$ in $\ur$ with respect to the metric
	\begin{align*}
		h_\reg=\no{	\Pi\du\dvol_{T\# V}^{h_\sing}}^{\frac{2}{d}}h_\sing.
	\end{align*}
	The subscript $\reg$ signifies the fact that the integral current $T\# V$ restricted to where $h_\reg,\phi_\reg$ are defined, i.e., $\ur,$ is completely smooth. This is in contrast with the base metric $h_\sing$ and the original calibration $\phi_\sing$, which are defined everywhere on $M,$ including $\sing (T\# V)=\sing(T+V).$
	
	We will rewrite $h_\reg,\phi_\reg$ in a fashion that is more useful for later constructions. The projection ${\Pi}$ provides a smooth orthogonal splitting of the tangent spaces to $M$ in $\ur$ into
	\begin{align}\label{split}
		\ker \ed\Pi\oplus \ker^\perp \ed\Pi.
	\end{align}
	Here $\oplus$ is the Riemannian direct sum of vector spaces, and $\ker^\perp \ed\Pi$ is the orthogonal complement of $\ker \ed\Pi.$ By definition of $\Pi,$ $\ker\ed\Pi$ are tangent spaces to the level sets of $\Pi,$ i.e., the submanifolds formed by geodesics orthogonal to $\tvu.$ Thus, $\kn$ restricted to $\tvu$
	is the tangent space to $\tvu$. Heuristically, one should think of $\kn$ as the natural extension of the tangent space of $\tvu$ into the tangent spaces of $\ur$.
	
	This also provides in $\ur$ a natural smooth splitting of the Riemannian metric $h_\sing$ into
	\begin{align*}
		h_\sing=(h_\sing)|_{\ker \ed\Pi}\oplus(h_\sing)|_{\ker^\perp \ed\Pi}.
	\end{align*}Here $\oplus$ means the Riemannian sum of the metrics on orthogonal subspaces. 
	
	With respect to the splitting (\ref{split}), the form ${\Pi}\du(\dvol_{T\#V})$ can be written as
	\begin{align*}
		{\Pi}\du\dvol_{T\#V}^{h_\sing}=\big({\Pi}\du(\dvol_{T\#V}^{h_\sing})(\ker^\perp \ed\Pi)\big)\big(\ker^\perp \ed\Pi\big)\du.
	\end{align*}Here $\ker^\perp \ed\Pi$ also denotes the unit simple $d$-vector representing $\ker^\perp \ed\Pi$ and
	$(\ker^\perp \ed\Pi)\du$ denotes the unit simple $d$-form Riemannian dual to the unit simple $d$-vector $\ker^\perp \ed\Pi.$ Then we have the following fact:
	\begin{fact}\label{calreg}
		The smooth closed form $$\phi_\reg={\Pi}\du\dvol^{h_\sing}_{T\#V}$$ is a calibration form defined on
		\begin{align}\label{ureq}		\ur\overset{\Ga\m}{\cong} S^{d-1}\times[-6,6]\times B_\mfr^{c},
		\end{align} that calibrates $ T\#V$ in	$\ur$, with respect to the metric
		$$h_\reg=(\phi_\reg(\ker^\perp \ed\Pi))^{\frac{2}{d}}h_{\sing}.$$
	\end{fact}
	\subsection{Calibrating $T\#V$ in $\ur$ by gluing calibrations}
	Now we have come to the meat of the matter. We will glue the metrics and calibrations we have obtained so far to make $T\#V$ calibrated everywhere. This subsection is taken from the proof of \cite[Theorem 4.6]{YZa}.   
	
	By (\ref{sdi}) and Assumption \ref{assumpcs}, $T\# V$ is calibrated by the smooth form $\phi_\sing$ in the smooth metric $h_\sing$ outside the subset $S^{d-1}\times[-\frac{1}{2},\frac{1}{2}]\times\{0\}$ of $\ur.$
	
	On the other hand, by Fact \ref{calreg}, $T\# V$ restricted to  $\ur$ is calibrated by the smooth form $\phi_\reg$ in the smooth metric $h_\reg$.
	
	Thus, by construction
	we have two Riemannian metrics $h_\reg,h_\sing$ and two smooth calibrations forms $\phi_\reg,\phi_\sing$, on
	\begin{align}\label{defnug}
		\ug\overset{\Ga}{\cong}S^{d-1}\times[-6,6]\times \ur\setminus S^{d-1}\times [-1,1]\times B_{\frac{\mfr}{5}}(\tvu),
	\end{align}
	so that $T\#V$ is calibrated in this region by the corresponding forms in the corresponding metrics, respectively. Now our goal is to glue these pairs of Riemannian metrics and calibration forms together. 
	The gluing is based on the following central fact:
	\begin{fact}\label{calk}
		Let $g$ be a positive definite quadratic form on $\R^{d+c}$, $\psi$ be a constant $d$-form on $\R^{d+c}$ and $\R^{d+c}=P\oplus Q$ be a direct sum decomposition of $\R^{d+c}$ into orthogonal complementary subspaces $P$ of dimension $d$ and $Q$ of dimension $c$, such that $\psi(P)>0.$ Then in the metric
		\begin{align*}
			h=(\psi(P))^{\frac{2}{d}}\big(g|_P\oplus (K g|_Q)\big),
		\end{align*}$\psi$ has comass
		\begin{align*}
			\cms_{h}\psi=1,
		\end{align*} 
		and calibrates $P,$ provided
		\begin{align}\label{kge}
			K\ge \binom{d+c}{d}(\psi(P))^{-1}\cms_g\psi.
		\end{align}
	\end{fact}
	Here $\psi(P)$ mean the value of $\psi$ at the simple unit $d$-vector representing $P$ with respect to $g.$
	The above fact is just a restatement of \cite[Lemma 3.4, Remark 3.5]{YZa}, which is in turn a restatement of a classical result on comass by Harvey-Lawson \cite[Lemma 2.14]{HLf}. 
	
	Roughly speaking, in the glued calibrations, the only component  that matters to us is $$\psi(\kn)(\kn)\du.$$ Then Fact \ref{calk} will enable us to basically annihilate the effects of other terms spanned by covectors in $\ker\ed\Pi\du$ in the glued calibration form. 
	\subsubsection{Notation convention}
	From now on, we will mostly our attention to the region $\ug$ defined in (\ref{defnug}).
	
	Recall that $h_\sing$ is defined throughout $M$ and $h_\reg$ is obtained by rescaling $h_\sing$ (Fact \ref{calreg}). 
	
	Thus, we will use $h_\sing$ as our base metric and use the symbols
	\begin{align*}
		h_{\ker \ed\Pi}=&(h_\sing)|_{\ker \ed\Pi}, \\h_{\ker^\perp \ed\Pi}=&(h_\sing)|_{\ker^\perp \ed\Pi}.
	\end{align*}
	Then by definition we have
	\begin{align*}
		h_\sing=\hkp \oplus \hkn.
	\end{align*}
	
	Also, for later definitions of auxiliary functions, we use the following variables
	\begin{defn}
		In $\ur$, define
		\begin{align*}
			r(p)=&\frac{5}{\mfr}\operatorname{dist}(p,\tvu),\\
			s(p)=&|\pi_{[-6,6]}\circ\Ga\m (p)|,
		\end{align*}
		where $\pi_{[-6,6]}$ is the canonical projection of the product $S^{d-1}\times[-6,6]\times B_\mfr^c$ onto $[-6,6]$.
	\end{defn}
	In other words, for a point $(q,t,b)\in S^{d-1}\times[-6,6]\times B_\mfr^c$ identifying via $\Ga$, we have
	\begin{align*}
		r(q,t,b)=&\frac{|b|}{\mfr/5},\\
		s(q,t,b)=&|t|.
	\end{align*}
	By construction, both $r$ and $s$ are smooth in $\ug$ (\ref{defnug}).
	
	In the subsequent gluing of calibrations and metrics, we need to divide $\ug$ into $5\times 4=20$ different parts according to the following twenty combinations of values of $r,s:$
	\begin{align*}
		&	r\in [1,2],r\in[2,3],r\in[3,4],r\in [4,5],\\
		&	s\in[1,2],s\in[2,3],s\in[3,4],s\in[4,5],s\in[5,6].
	\end{align*}
	To simplify the notations, 
	\begin{defn}\label{defnuij}
		For integers $0\le i\le 5,0\le j\le 4,$ we define
		\begin{align*}
			U^j=&r\m[j,j+1]\\
			U_i=&s\m[i,i+1].
		\end{align*}
	\end{defn}
	In other words, 
	\begin{align*}
		U^j=&S^{d-1}\times [-6,6]\times \ov{B_{\frac{(j+1)}{5}\mfr}^c\setminus B_{\frac{j}{5}\mfr}^c},\\
		U_i=&	S^{d-1}\times ([-i-1,-i]\cup[i,i+1])\times B_{\mfr}^c.
	\end{align*}
	\newcommand{\ua}{U_1}\newcommand{\ub}{U_2}\newcommand{\uc}{U_3}\newcommand{\ud}{U_4}
	For instance, $\ug$ as defined in (\ref{defnug}) can be written as
	\begin{align*}
		\ug=\ur\setminus(U^0\cap U_0).
	\end{align*}
	\subsubsection{Anti-derivatives of forms}
	Using Poincar\'{e} lemma (\cite{PL}) on the homotopy $H_\tau$ with time parameter $\tau$ on $\ur,$ defined by
	\begin{align*}
		H_\tau(q,t,b)=(q,t,\tau b),
	\end{align*}
	we can integrate the closed $d$-forms $\phi_\sing,\phi_\reg$ to find
	smooth $(d-1)$-form $\Phi_\sing$ on $\ur$ such that
	\begin{align*}
		d\Phi_\sing=&\phi_\sing,\\
		d\Phi_{\reg}=&\phi_\reg.
	\end{align*}
	and
	\begin{align}\label{eqantiz}
		\Phi_\sing\equiv\Phi_\reg\textnormal{ on }\tvu\cap \ug.
	\end{align}
	\subsubsection{Auxiliary functions}\label{auxf}\begin{figure}
		\centerline{
			\includegraphics[width=0.9\paperwidth]{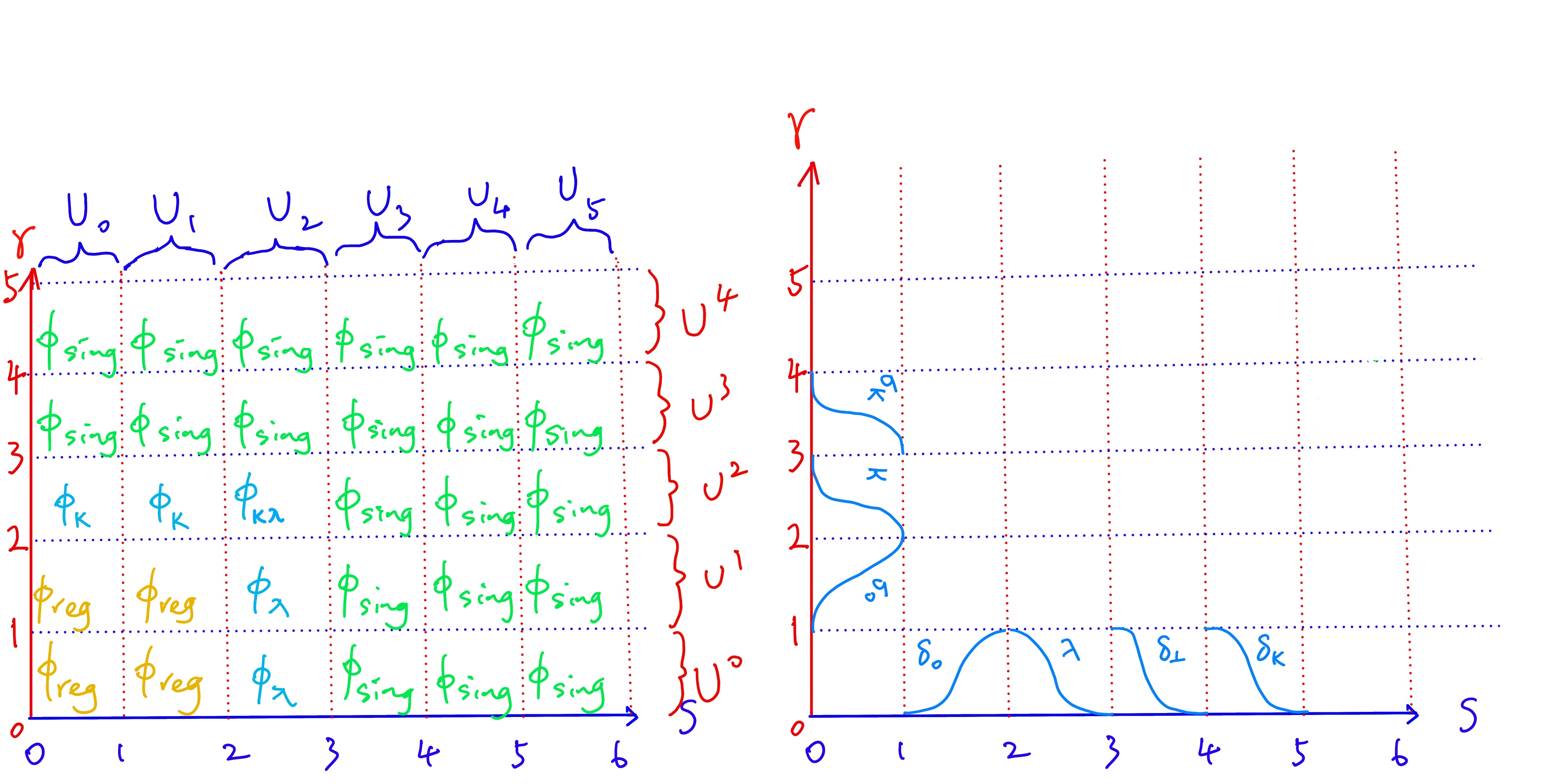}}
		\caption{Definition of $\phi_{\ka\lam}$ and the auxiliary functions}
		\label{fignbd}
	\end{figure}
	Let $\ai:\R\to\R$ be a smooth monotonic function that is $0$ on $(-\infty,0]$ and $1$ on $[1,\infty)$. 
	\begin{defn}\label{defnaux}
		Define
		\begin{align*}
			\de_0=&\ai\circ(s-1),\\
			\lam=&\ai\circ(-s+3),\\
			\de_\perp=&\ai\circ(-s+4),\\
			\de_K=&\ai\circ(-s+5),\\
			\si_0=&\ai\circ(r-1),\\
			\ka=&\ai\circ (-r+3),\\
			\si_K=&\ai\circ(-r+4).
		\end{align*}
	\end{defn}
	Then $\de_0,\lam,\de_\perp,\de_K,\si_0,\ka,\si_K$ are smooth functions defined on $\ug$ with values between $[-1,1]$, and not equal to $0,1$ only on $\ua,\ub,\uc,\ud,U^1,U^2,U^3$ respectively. For an intuitive illustration, please refer to Figure \ref{figgluing}.
	\subsubsection{Gluing of forms}
	\begin{defn}
		In $\ur$, define
		\begin{align}
			\phi_\lam=&	d\big((1-\lam)\Phi_\sing+\lam\Phi_\reg\big)=(1-\lam)\phi_\sing+\lam\phi_\reg+d\lam\w (-\Phi_\sing+\Phi_\reg),\\
			\phi_{\ka}=&d\big((1-\ka)\Phi_\sing+\ka\Phi_\reg\big)=(1-\ka)\phi_\sing+\ka\phi_\reg+d\ka\w (-\Phi_\sing+\Phi_\reg),\\
			\label{defnpkl}	\phi_{\ka\lam}=&d\big((1-\ka\lam)\Phi_\sing+\ka\lam\Phi_\reg\big)=(1-\ka\lam)\phi_\sing+\ka\lam\phi_\reg+d(\ka\lam)\w (-\Phi_\sing+\Phi_\reg).	
		\end{align}
		Then define
		\begin{align}\label{defp}	
			\phi=
			\begin{cases}
				\phi_\sing,&\text{on }B_\mfr(\tvu)\cp,\\
				\phi_{\ka\lam},&\text{on }\ug,\\
				\phi_\reg,&\text{on }U_0\cap U^0.
			\end{cases}
		\end{align}
	\end{defn}It is straightforward to verify that
	\begin{fact}\label{fctpp}
		We have\begin{align}\label{defpp}	
			\phi=
			\begin{cases}
				\phi_\sing,&\text{on }B_\mfr(\tvu)\cp\cup\{r\ge 3\}\cup\{s\ge 3\},\\
				\phi_\reg,&\text{on }\{r,s\le 2\},\\
				\phi_\ka,&\text{on }\{2\le r\le 3,s\le 2\},\\
				\phi_\lam,&\text{on }\{ r\le 2,2\le s\le 3\},\\
				\phi_{\ka\lam},&\text{on }\{2\le r\le 3,2\le s\le 3\}.		
			\end{cases}
		\end{align}
	\end{fact}
	In other words, $\phi$ does not equal to $\phi_\sing$ or $\phi_\reg$ only on $\{2\le \max\{r,s\}\le 3\}.$
	
	Using Fact (\ref{fctpp}), it is straightforward to verify that
	\begin{fact}
		The form $\phi$ is closed smooth form defined on $M.$  
	\end{fact}	
	The reader should now consult Figures \ref{fignbd}.
	\subsubsection{Gluing of metrics}
	In the following, when we write $\ker^\perp \ed\Pi$ as the input of a differential form, we always mean the unit simple $d$-vector corresponding to $\kn$ in the base metric $h_\sing.$ Similarly, when we write $(\kn)\du$, we always mean the unit simple $d$-vector Riemannian dual to the unit simple $d$-vector $\kn$ in the base metric $h_\sing.$
	\begin{fact}\label{fctp1}
		We have
		\begin{align*}
			\phi(\kn)\equiv \phi_\lam(\kn)\equiv 1,
		\end{align*}on $\tvu.$
	\end{fact}
	For proof of the above fact, the reader should first recall  Fact \ref{fctpp} and consult Figure \ref{fignbd} for intuition. We have $\phi_\reg(\kn)=1$ on $\tvu$ by definition of $\phi_\reg$ in Fact \ref{calreg}. Moreover, $\phi_\sing(\kn)=1$ on $\tvu\cap U_0\cp$ by (\ref{sdi}). Now use (\ref{eqantiz}) and (\ref{defnpkl}), we have $\phi(\kn)=\phi_\lam(\kn)=1$ on $\ug.$ Combining all three cases we obtain Fact \ref{fctp1}.
	
	An important thing to note is that $\phi_\lam$ is independent of the value of $\mfr$ we choose in (\ref{mfr0}), which is straightforward to verify. In other words, $\phi_\lam$ produced from $\mfr=\mfr_0$ restricts to $\phi_\lam$ produced from any value of $0<\mfr<\mfr_0$. Now, let us consider $\phi_\lam$ produced from $\mfr=\mfr_0.$ By Fact \ref{fctp1}, $\phi_\lam$ restricted to $\{r\le\e\}$ with $\e$ very small will satisfy 	$$2\ge\phi_\lam(\kn)\ge \frac{1}{2}.$$ Thus, if we switch to $\mfr=\e\mfr_0,$ we have $2\ge \phi_\lam(\kn)\ge \frac{1}{2}$ everywhere in $\ur$ with $\mfr=\e\mfr_0.$ 
	
	We fix this choice of $\mfr=\e\mfr_0$. Since $\phi=\phi_\lam$ on $U^0\cup U^1$ by definition of $\phi$ and $\ka,$ we deduce that
	\begin{assump}\label{philbd}
		\begin{align}
			2\ge	\phi(\kn)\ge \frac{1}{2}\textnormal{ on }U^0\cup U^1.
		\end{align}  
	\end{assump}
	\begin{defn}
		Set
		\begin{align}
			\label{defk}	K=&\binom{d+c}{d}\frac{\cms_{h_\sing}\phi|_{\ur}}{\inf_{p\in U^0\cup U^1}\phi_p(\kn)}.
		\end{align}
	\end{defn}
	Note that by construction, we always have
	\begin{align}\label{kineq}
		K\ge {\cms_{h_\sing}\phi|_{\ur}}\ge 1
	\end{align}
	Now we are ready to define our metric $h.$ 
	
	Let us first define two auxiliary Riemannian metrics $\hgd$ and $\hpp$ on $U^0\cup U^1$  which take advantage of (\ref{philbd}),
	\begin{align}h_{\operatorname{glued}}=
		\begin{cases}
			h_\sing=\hkn\oplus \hkp,&\text{on }(U^0\cup U^1)\cap U_5,\\
			\hkn\oplus \big((1+K\de_K)\hkp\big),&\text{on }{(U^0\cup U^1)\cap \ud},\\
			\bigg((1-\de_\perp)+\de_\perp(\phi(\kn))\bigg)^{\frac{2}{d}}\Big(\hkn\oplus \big((1+K)\hkp\big)\Big),&\text{on }{(U^0\cup U^1)\cap \uc},\\
			(\phi(\kn)^{\frac{2}{d}})\Big(\hkn\oplus \big((1+K)\hkp\big)\Big),&\text{on }{(U^0\cup U^1)\cap \ub},\\
			(\phi(\kn)^{\frac{2}{d}})\Big(\hkn\oplus \big((1+\de_0 K)\hkp\big)\Big),&\text{on }{(U^0\cup U^1)\cap \ua},\\
			h_\reg=(\phi(\kn)^{\frac{2}{d}})\big(\hkn\oplus \hkp\big),&\text{on }(U^0\cup U^1)\cap U_0.				
		\end{cases}
	\end{align}\begin{figure}
		\centerline{
			\includegraphics[width=0.9\paperwidth]{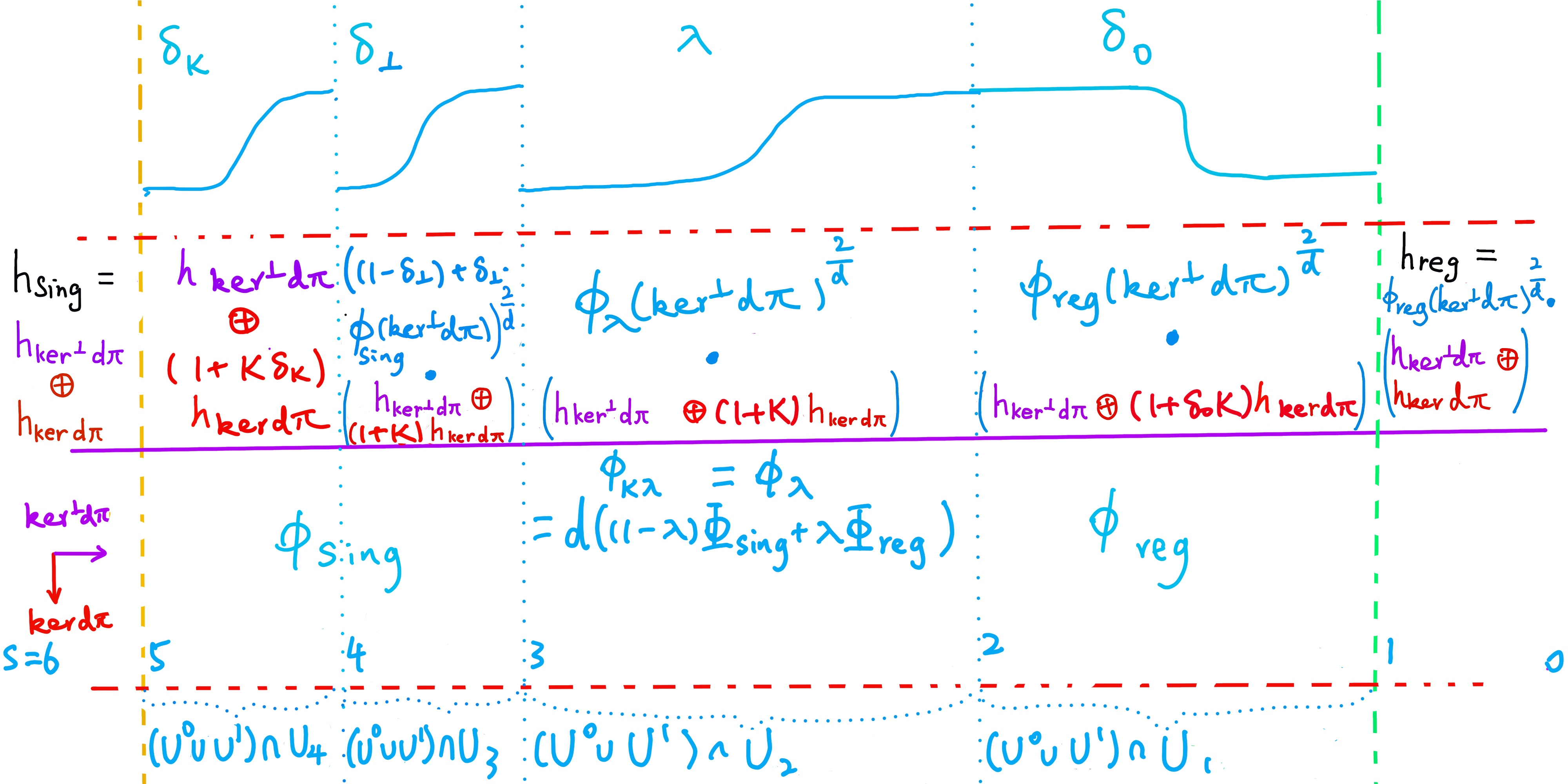}}
		\caption{Gluing of forms and the change of metrics in $U^0\cup U^1$}
		\label{figgluing}
	\end{figure}
	\begin{align}\hpp=
		\begin{cases}
			h_\sing=\hkn\oplus \hkp,&\text{on }(U^0\cup U^1)\cap U_5,\\
			\hkn\oplus \big((1+K\de_K)\hkp\big),&\text{on }{(U^0\cup U^1)\cap \ud},\\
			\Bigg((1-\de_\perp)+\de_\perp\Big((1-\si_0)\phi(\kn)+\si_0 K\Big)\Bigg)^{\frac{2}{d}}\\\textbf{ }\textbf{ }\textbf{ }\textbf{ }\cdot\Big(\hkn\oplus \big((1+K)\hkp\big)\Big),&\text{on }{(U^0\cup U^1)\cap \uc},\\
			\bigg((1-\si_0)\phi(\kn)+\si_0 K\bigg)^{\frac{2}{d}}\Big(\hkn\oplus \big((1+K)\hkp\big)\Big),&\text{on }{(U^0\cup U^1)\cap \ub},\\
			\bigg((1-\si_0)\phi(\kn)+\si_0 K\bigg)^{\frac{2}{d}}\Big(\hkn\oplus \big((1+\de_0 K)\hkp\big)\Big),&\text{on }{(U^0\cup U^1)\cap \ua},\\
			\bigg((1-\si_0)\phi(\kn)+\si_0 K\bigg)^{\frac{2}{d}}\big(\hkn\oplus \hkp\big),&\text{on }(U^0\cup U^1)\cap U_0.				
		\end{cases}
	\end{align}
	Assumption (\ref{philbd}) allows us to ensure that $\phi(\kn)^{\frac{2}{d}}$ is smooth and well-defined. It is straightforward to verify that $\hpp$ restricts to $\hgd$ on $U^0$ using Definition \ref{defnaux} and $\hpp\ge \hgd$ as quadratic forms on $U^0\cup U^1$. Technically speaking, only $\hpp$ is needed here, but introducing $\hgd$ will significantly simplify the our reasoning and essentially all other parts of the wanted metric $h$ in $\ur$ are based on modifications of $\hgd.$

	On $U^2\cup U^3$, we define two metrics $\htt$ and $\hmt,$
	\begin{align}\htt=
		\begin{cases}
			h_\sing=\hkn\oplus \hkp,&\text{on }(U^2\cup U^3)\cap U_5,\\
			\hkn\oplus \big((1+K\de_K)\hkp\big),&\text{on }{(U^2\cup U^3)\cap\ud},\\
			\big((1-\de_\perp)+\de_\perp K\big)^{\frac{2}{d}}\Big(\hkn\oplus \big((1+K)\hkp\big)\Big),&\text{on }{(U^2\cup U^3)\cap\uc},\\
			K^{\frac{2}{d}}\Big(\hkn\oplus \big((1+K)\hkp\big)\Big),&\text{on }{(U^2\cup U^3)\cap\ub},\\
			K^{\frac{2}{d}}\Big(\hkn\oplus \big((1+\de_0 K)\hkp\big)\Big),&\text{on }{(U^2\cup U^3)\cap\ua},\\
			K^{\frac{2}{d}}\big(\hkn\oplus \hkp\big),&\text{on }(U^2\cup U^3)\cap U_0.				
		\end{cases}
	\end{align}
	\begin{align}\hmt=\si_K \htt+(1-\si_K)h_\sing.
	\end{align}
	By definition, $\hmt$ restricts to $\htt$ on $U^2.$
	\begin{defn}\label{defnh}
		Define a metric $h$ on $M$ by setting
		\begin{align*}
			h=\begin{cases}
				h_\sing,&\textnormal{on }(\ur)\cp,\\
				h_\sing,&\textnormal{on }U^4,\\
				\hmt	,&\textnormal{on }U^2\cup U^3,\\
				\hpp	,&\textnormal{on }U^0\cup U^1,		\end{cases}
		\end{align*}
	\end{defn}
	\begin{lem}\label{hgdc}
		The metric $h$	is a smooth Riemannian metric on $M$ and $\phi$ defined in (\ref{defp}) is a calibration form in metric $h$ that calibrates $T\#V$. 
	\end{lem}
	Note that the above lemma verifies the last bullet of Lemma \ref{lemcs}, the only bullet not proved yet. Thus to finish our proof of Lemma \ref{lemcs}, the rest of this section is devoted to the proof of the above lemma.
	
	The reader should consult Figures \ref{fignbd} and \ref{figgluing} for intuition. For later proof, it is beneficial to write $h$ in a different form
	\begin{align}\label{defnhc}
		h=\begin{cases}
			h_\sing,&\textnormal{on }(\ur)\cp,\\
			h_\sing,&\textnormal{on }U^4,\\
			\hmt	,&\textnormal{on }U^3,\\
			\htt	,&\textnormal{on }U^2,\\
			\hpp	,&\textnormal{on }U^1,\\
			\hgd	,&\textnormal{on }U^0.		\end{cases}
	\end{align}
	\subsection{Smoothness of $h$}
	Let us verify the smoothness of $h.$
	
	To start, let use prove $h$ is smooth in $\ur.$
	
	It is straightforward to check that $\hgd$ is a smooth positive Riemannian metric $U^0\cup U^1,$ using Definition \ref{defnaux} and Assumption \ref{philbd}. 
	
	Since $\hpp$ is a smooth modification of $\hgd$ using $\de_0$, we see that $\hpp$ is a smooth positive definite Riemannian metric on $U^0\cup U^1$. 
	
	In the same vein, we can check that $\htt,\hmt$ on $U^2\cup U^3$ are both smooth positive definite Riemannian metrics.
	
	It is straightforward to check that $\htt$ patches smoothly onto $\hpp$ at the interface $\{s=2\}$. In summary, $h$ is indeed a smooth metric on $\ur.$ To prove that $h$ is smooth on $M,$ we only need to verify that $h$ patches smoothly to $h_\sing$ on $\pd\ur.$ This follows directly from the definition of $h$ in $ U^4\cup U_5.$
	
	To sum it up, we have verified that $h$ is smooth. \subsection{Comass of $\phi$ in $\hgd$}
	As can be seen from the respective definitions, the metric $\hgd$ serves as a model for all other parts of metrics $\hpp,\htt,\hmt$. Let us first verify that
	\begin{fact}\label{fcthgd}
		The form $\phi$ is a calibration form in metric $\hgd$ on $U^0\cup U^1$ and calibrates $\tvu$.
	\end{fact}
	The reader should consult Figure \ref{figgluing} for intuition
	
	Note that $\phi\equiv\phi_\sing,\hgd\equiv h_\sing$ on $(U^0\cup U^1)\cap U_5$ and $\phi\equiv \phi_\reg,\hgd\equiv h_\reg$ on $(U^0\cup U^1)\cap U_0$. Thus, by Assumption \ref{assumpcs}, (\ref{sdi}), and Fact \ref{calreg}, $\phi$ is a calibration and calibrates $\tvu$ with respect to the metric $\hgd$ in $(U^0\cup U^1)\cap (U_0\cup U_5)$.
	
	On $(U^0\cup U^1)\cap {\ud},$ we have $$\hgd=\hkn\oplus \big((1+K\de_K)\hkp\big)\ge \hkn\oplus \hkp=h_\sing,$$ as quadratic forms and $\phi=\phi_\sing.$ By bullet (\ref{cms1}) in Fact \ref{cmsvec}, the form $\phi$ remains a calibration on $(U^0\cup U^1)\cap {\ud}.$ When restricted to $(U^0\cup U^1)\cap \tvu,$ the metrics $\hgd$  and $h_\sing$ are equal to each other. Thus, we deduce that in $(U^0\cup U^1)\cap {\ud}$, the form $\phi=\phi_\sing$ still calibrates $\tvu$  with respect to $\hgd.$
	
	On $(U^0\cup U^1)\cap {\uc}$, first note that $\frac{1}{2}\le\phi_\sing(\ker^\perp\ed\Pi)\le 1$ by Assumption \ref{philbd} and Assumption \ref{assumpcs}. Thus,  on $(U^0\cup U^1)\cap{\uc}$ we have
	\begin{align*}
		\hgd=&	\bigg((1-\de_\perp)+\de_\perp(\phi_\sing(\kn))\bigg)^{\frac{2}{d}}\Big(\hkn\oplus \big((1+K)\hkp\big)\Big)\\\ge& \big(\phi_\sing(\ker^\perp\ed\Pi)^{\frac{2}{d}}\big)\big(\hkn\oplus((1+K)\hkp)\big)\\
		\ge& \big(\phi_\sing(\ker^\perp\ed\Pi)^{\frac{2}{d}}\big)\big(\hkn\oplus(K\hkp)\big),
	\end{align*}as quadratic forms.
	By bullet (\ref{cms1}) in Fact \ref{cmsvec}, inequality (\ref{kge}) of Fact \ref{calk}, and (\ref{defk}), we deduce that on $(U^0\cup U^1)\cap{\uc}$, the form $\phi$ remains a calibration in $\hgd.$ Again, the restrictions of $h_\sing$ and $\hgd$ to $\uc\cap \tvu$ are equal, so $\tvu$ is calibrated by $\phi=\phi_\sing$ in $(U^0\cup U^1)\cap{\uc}$. 
	
	On $(U^0\cup U^1)\cap{\ub},$ we have
	\begin{align*}
		\hgd=&(\phi_{\ka\lam}(\kn)^{\frac{2}{d}})\Big(\hkn\oplus \big((1+K)\hkp\big)\Big)\\\ge& \big(\phi_{\ka\lam}(\ker^\perp\ed\Pi)^{\frac{2}{d}}\big)\big(\hkn\oplus(K\hkp)\big),
	\end{align*}as quadratic forms.
	By bullet (\ref{cms1}) in Fact \ref{cmsvec}, inequality (\ref{kge}) of Fact \ref{calk}, and (\ref{defk}), we deduce that on ${\ub}$, the form $\phi$ is a calibration form in $\hgd.$ Also when restricted to ${\ub}\cap \tvu$, we have $\hgd=h_\reg$, and $\phi=\phi_\reg$. This implies that on $(U^0\cup U^1)\cap{\ub},$ the form $\phi$ calibrates $\tvu$ with respect to $\hgd.$
	
	Finally, on $(U^0\cup U^1)\cap\ua$, we have \begin{align*}
		\hgd=&	(\phi_\reg(\kn)^{\frac{2}{d}})\Big(\hkn\oplus \big((1+\de_0 K)\hkp\big)\Big)\\\ge&(\phi_\reg(\kn)^{\frac{2}{d}})\big(\hkn\oplus\hkp\big)\\=& h_\reg.
	\end{align*} as quadratic forms. By bullet (\ref{cms1}) of Fact \ref{cmsvec}, this means that $\phi=\phi_\reg$ is a calibration in ${\ua}$ with respect to $\hgd.$ Again, the restrictions of $\hgd$ and $h_\reg$ to $\ua\cap  \tvu$ are equal, so we deduce that $\tvu$ is calibrated by $\phi=\phi_\reg$ in $\ua$.  
	\subsection{Comass of $\phi$ in $h$}
	Let us use the representation (\ref{defnhc}). By (\ref{kineq}) on $U^0\cup U^1$,  we have, as quadratic forms,
	$$\hpp\ge \hgd.$$
	Thus, by bullet (\ref{cms1}) of Fact \ref{cmsvec} and Fact \ref{fcthgd}, we deduce that $\phi$ is calibration form in $U^0\cup U^1$ with respect to $\hpp.$
	
	On $U^2\cap(U_0\cup U_1\cup U_2)$, we have, as quadratic forms,
	$$\htt\ge K^{\frac{2}{d}} h_\sing.$$ By (\ref{defk}) and bullet (\ref{cms1}) (\ref{cms2}) of Fact \ref{cmsvec} this implies that $\phi$ is a calibration form with respect to the metric $\htt$ on $U^2\cap(U_0\cup U_1\cup U_2)$. On $U^2\cap(U_3\cup U_4\cup U_5)$, by Fact (\ref{defpp}), we have as quadratic forms,
	\begin{align*}
		\htt\ge h_\sing,
	\end{align*}and we have the restriction
	\begin{align*}
		\phi=\phi_\sing.
	\end{align*}
	Thus by (\ref{defnpsing}) of Assumption \ref{assumpcs} and bullet (\ref{cms1}) of Fact \ref{cmsvec}, we deduce that $\phi$ is a calibration form on $U^2\cap(U_3\cup U_4\cup U_5)$ with respect to the metric $\htt.$ 
	
	On $U^3,$ by Fact (\ref{defpp}), we have 
	\begin{align*}
		\phi=\phi_\sing,
	\end{align*}
	and as quadratic forms,
	\begin{align*}
		\hmt\ge h_\sing.
	\end{align*}
	Thus by (\ref{defnpsing}) of Assumption \ref{assumpcs} and bullet (\ref{cms1}) of Fact \ref{cmsvec}, we deduce that $\phi$ is a calibration form on $U^3$ with respect to the metric $\hmt.$ 
	
	Since $h,\phi$ coincides with $h_\sing,\phi_\sing$ on $U^4$ and $(\ur)\cp,$ we deduce that $\phi$ is a calibration form on $M$ with respect to the metric $h$ and calibrates $T$ in $(\ur)\cp$.
	
	On other hand, since $h=\hgd$ on $U^0$, by Fact \ref{fcthgd}, we deduce that $\phi$ calibrates $\tvu$ in $h.$ 
	
	We are done with proving Lemma \ref{hgdc}, and thus have finished the proof of Lemma \ref{lemcs}.
	\section{Calibrating currents in connected sums of ambient manifolds}\label{seccs2}
	In order to prove Theorem \ref{thmcs}, we also need the following lemma besides Lemma \ref{lemcs}. We assume the reader is familiar with the standard connected sums of manifolds and boundary connected sums of manifolds with non-trivial boundaries, and the reader can refer to \cite[Chapter VI.1 and VI.3]{AK} for details. Let us first deal with standard connected sums, then boundary connected sums and finally wrap up the proof of Theorem \ref{thmcs}.
	\subsection{Standard connected sums}
	\begin{assump}\label{assumpcsa}Assume that
		\begin{enumerate}
			\item $T,T'$ are two $d$-dimensional area-minimizing integral currents calibrate by $d$-forms $\psi,\psi'$ in closed connected $(d+c)$-dimensional Riemannian manifolds $(M,g),(M',g'),$ respectively. 
			\item $B,B'$ are	standard $(d+c)$-dimensional open balls in $M,M'$, whose closures are disjoint from the supports of $T,T'$, on $M,M'$, respectively.
			\item $M\# M'$ is the differential topology connected sum of $M,M'$ obtained by replacing the disjoint union $B\cup B'$ with a neck $S^{d+c-1}\times[-1,1].$
			\item $i,i'$ are the natural embeddings of $M\setminus B,M'\setminus B'$ into $M\# M'$, respectively.
	\end{enumerate}\end{assump}
	\begin{lem}\label{lemcsa}
		Under Assumption \ref{assumpcsa}, there exists a smooth metric $h$ and a smooth closed $d$-form $\phi$ on $M\# M'$, such that
		\begin{enumerate}
			\item $\phi$ is a calibration form on $(M\# M',h)$.
			\item $\phi$ calibrates the integral current $i\pf(T)+i'\pf(T').$
		\end{enumerate}
	\end{lem}
	\begin{proof}
		Using Assumption \ref{assumpcsa}, equip  $i(M\setminus B)$ with metric $i\pf g$ and $\imbb$ with metric  $i'\pf g'$, then extend $i\pf g,i'\pf g'$ smoothly to $M\# M'\setminus(\imb\cup\imbb)$. In other words we obtain a metric $$\ov{g}$$ on $M\# M'$, such that
		\begin{align*}
			\ov{g}=\begin{cases}
				i\pf g,&\textnormal{on }\imb,\\
				i'\pf g',&\textnormal{on }\imbb.
			\end{cases}
		\end{align*}On $\imb\cup \imbb$ We already  a calibration form that calibrates $i\pf(T)+i'\pf(T')$ in $\ov{g}$, which are just the disjoint union of $i\pf \psi,i'\pf\psi'.$

		Our goal is to glue the forms $i\pf\psi,i'\pf\psi'$ together in the neck region of the connected sum.
		
		First, we need to make space for the gluing. It is straightforward to verify the following. There exists a closed domain $$U$$ of $M\# M'$, diffeomorphic via a diffeomorphism $\Lambda$ to
		\begin{align*}
			U\overset{\Lambda}{\cong} S^{d+c-1}\times[-3,3],
		\end{align*}
		such that
		\begin{align}
			i(M\setminus B)\cap U\overset{{\Lambda}}{\cong}& S^{d+c-1}\times[-3,-1],\\
			i'(M'\setminus B')\cap U\overset{\Lambda}{\cong}& S^{d+c-1}\times[1,3],\\
			\label{suppuu}(\supp i\pf T\cup \supp i'\pf T')\cap U=&\es.
		\end{align}
		Thus, we have
		\begin{align*}
			H^d(i(M\setminus B)\cap U,\R)=0,\\
			H^d(i'(M'\setminus B')\cap U,\R)=0.
		\end{align*}
		By de Rham's theorem, this implies that we can find smooth $(d+1)$-dimensional forms $\Psi,\Psi'$ on $\imb\cap U,\imbb\cap U,$ respectively, such that
		\begin{align*}
			d\Psi=\psi,\textbf{ }d\Psi'=\psi'.
		\end{align*}
		Let $\ai:\R\to\R$ be a smooth monotonic function that is $0$ on $(-\infty,0]$ and $1$ on $[1,\infty)$. 	Define auxiliary functions
		\begin{align*}
			\tau_{-}=&\ai\circ(-\pi_{[-3,3]}\circ\Lambda -1),\\
			\tau_+=&\ai\circ(\pi_{[-3,3]}\circ\Lambda -1),\\
			\zeta=&\ai\circ(|\pi_{[-3,3]}\circ\Lambda|-2)
		\end{align*}a form $\phi$ on $M$ as follows
		\begin{align*}
			\phi=\begin{cases}
				\psi,&\textnormal{on }\imb\setminus U,\\
				d\big(\tau_-\Psi\big),&\textnormal{on }\Lambda\m (S^{d-1}\times(-3,-1)),\\
				0,&\textnormal{on }\Lambda\m(S^{d-1}\times[-1,1]),\\
				d\big(\tau_+\Psi'\big),&\textnormal{on }\Lambda\m (S^{d-1}\times(1,3)),\\	\psi',&\textnormal{on }\imbb\setminus U.
			\end{cases}
		\end{align*}	It is straight forward to verify that $\phi$ is a smooth form on $M\# M'$. Set
		\begin{align*}
			L=\cms_{\ov{g}}\phi.
		\end{align*} Define a metric
		\begin{align*}
			h=\begin{cases}
				\ov{g},&\textnormal{on }\imb\setminus U,\\
				\big(\zeta+(1-\zeta)L\big)^{\frac{2}{d}}\ov{g},&\textnormal{on }\Lambda\m (S^{d-1}\times(-3,-3)),\\	\ov{g},&\textnormal{on }\imbb\setminus U.
			\end{cases}
		\end{align*}
		It is straightforward to verify that $h$ is a smooth metric.
		Using bullets (\ref{cms1}) (\ref{cms2}) of Fact \ref{cmsvec}, a direct calculation shows that $\phi$ is a calibration form in $h$. By (\ref{suppuu}), we deduce that $\phi$ calibrates $i\pf(T)+i\pf(T')$ on $(M,h).$
	\end{proof}
	\subsection{Boundary connected sums}
	Analogues of Lemma \ref{lemcsa} also holds for boundary connected sums of manifolds.
	\begin{assump}\label{assumpcsb}Assume that
		\begin{enumerate}
			\item $T,T'$ are two $d$-dimensional area-minimizing integral currents calibrated by $d$-forms $\psi,\psi'$ in the interior of $(M,g),(M',g'),$ two connected $(d+c)$-dimensional Riemannian manifolds with non-trivial boundaries $\pd M$ and $\pd M'$, respectively.
			\item $D,D'$ are	standard $(d+c-1)$-dimensional closed balls in $\pd M,\pd M'$, respectively. 
			\item  $M\prescript{}{\pd}{\#} M'$ is the differential topology boundary connected sum of $M,M'$ obtained by gluing the two ends $B^{d+c-1}\times\{-1\}$ and $B^{d+c-1}\times\{1\}$ of a neck $B^{d+c-1}\times[-1,1]$ to $D$ and $D'$, respectively
			\item $i,i'$ are the natural embeddings of $M\setminus D,M'\setminus D'$ into $M\prescript{}{\pd}{\#} M'$, respectively.
	\end{enumerate}\end{assump}
	\begin{lem}\label{lemcsb}
		Under Assumption \ref{assumpcsb}, there exists a smooth metric $h$ and a smooth closed $d$-form $\phi$ on $M\prescript{}{\pd}{\#} M'$, such that
		\begin{enumerate}
			\item $\phi$ is a calibration form on $(M\prescript{}{\pd}{\#} M',h)$.
			\item $\phi$ calibrates the integral current $i\pf(T)+i'\pf(T').$
		\end{enumerate}
	\end{lem}
	\begin{proof}
		The proof is the same as the proof of Lemma \ref{lemcsa}, replacing $B,B'$ by $D,D'$,  $S^{d+c-1}\times[-1,1]$ by $B^{d+c-1}\times[-1,1]$ and $M\#M'$ by $M\prescript{}{\pd}{\#}M'$.
	\end{proof}
	Boundary connected sums come into our proof in the next section via the following lemma.
	\begin{lem}\label{lembcsd}
		If $W,W'$ are two smooth domains of $M,$ and $\ga$ is a smooth curve from $w\in \pd W$ to $w'\in \pd W'$ such that $\ga$ meets $W,W'$ only at $w,w'$, is transverse to $\pd W,\pd W',$ and is orientation reversing. Then for any neighborhood $U(\ga)$ of $\ga,$ there exists a smooth embedding $\Ga$ of $W\bcs W'$ into $M$ and $\Ga(W\bcs W')\setminus (W\cup W')$ is diffeomorphic to a neck $B_1^{d+c-1}\times[-1,1]$ contained in $U(\ga).$
	\end{lem}
	\begin{proof}
		The proof is similar to the construction in Section \ref{seccht}. First take a coordinate chart $U$ with coordinate labels $(x_1,\cd,x_{d+c})$, so that $W$ is the half space $x_{d+c}\le -\frac{1}{2}$, $W'$ is the half space $x_{d+c}\ge \frac{1}{2}$ and $\ga$ is the line segment from $(0,\cd,0,-\frac{1}{2})$ and $(0,\cd,0,\frac{1}{2}).$ Glue a neck $B_\e^{d-1}(\{0\}^{d-1})\times\{0\}^c\times [-\frac{1}{2},\frac{1}{2}]$ for $\e$ small and smooth the corners. We are done.
	\end{proof}
	\subsection{Proof of Theorem \ref{thmcs}}
	This follows directly from Lemma \ref{lemcs}, Lemma \ref{lemcsa} and Lemma \ref{lemcsb}.
	\section{Proof of Theorem  \ref{thmi}}\label{secpf}
	A brief plan for our proof of Theorem \ref{thmi} is as follows. First we will alter the topological representative of $[\Si]$ to add our desired singular sets, while establishing Features \ref{ft0} and \ref{ft1} of our plan in Section \ref{secsk}. A tricky part is to ensure the conditions in Lemma \ref{lemzhang} holds. Then we will use Lemmas \ref{lemzhang}, \ref{lemcs}, \ref{lemcsa}, and \ref{lemcsb} to find a smooth Riemannian metric in which the altered representative produced is area-minimizing, thus achieving Feature \ref{ft2} in Section \ref{secsk}.
	\subsection{Preparing an embedding}
	Recall Assumptions \ref{assumpmff} and \ref{assumpbs}. 
	
	Fix a point $s$ of $\Si$. We want to emphasize that in this subsection, we will not deal with Riemannian geometrical properties of the objects involved and we use an ambient metric on $M$ only as a quantitative measure of ambient topological properties. Without loss of generality, fix a smooth Riemannian metric $h $ on $M$ such that,
	\begin{assump}We have\label{assumpcb}
		\begin{enumerate}
			\item $B(s)$ is the radius $1$ geodesic ball around $s$,
			\item $B(s)$ is geodesically convex, i.e., the shortest distance between any two points in $B(s)$ is always realized by a geodesic in $B(s),$
			\item the metric $g$ is flat when restricted to $B(s)$
			\item in the normal coordinate system $(x_1,\cd,x_{d+c})$ on $B(s)$, $\Si$ restricted to $B(s)$ equals the coordinate plane $x_1\cd x_d.$  		 
		\end{enumerate}
	\end{assump} 
	This can be done, e.g., by first adopting a coordinate system near $p$ such that $\Si$ is a coordinate plane and then extending the standard coordinate flat metric of $B(s)$ to a smooth metric $h$ on $M$.
	\begin{defn}
		Define the smooth product manifold 
		\begin{align*}
			P_N=M_N\times B_1^{c-(d-\dim N+1)}.
		\end{align*}
	\end{defn}Here we regard $B_1^0$ as a point and $M_N\times B_1^0=M_N.$
	Let us first prove that
	\begin{fact}
		There exists a smooth embedding $\Ga$ of the boundary connected sum $$\prescript{}{\pd}{\#}_{N\in\mathcal{F}}P_N,$$ into $B(s)\cap \{x_{d}>0\}$.
	\end{fact}
	\begin{proof}
		By Fact \ref{fctnk}, for each $N\in\mathcal{F},$ there exists an embedding of $M_N$ into $B_1^{d+c}$ with trivial normal bundle. Thus a closed tubular neighborhood of this embedding of $M_N$ is a smooth embedding of $P_N$ into
		$B_1^{d+c}.$ Use $\Ga_N$ to denote this embedding.
		
		Fix an arbitrary order on $\mathcal{F}$. In other words, we list the elements of $\mathcal{F}$ as $N_1,\cd,N_{|\mathcal{F}|}$, where $|\mathcal{F}|$ is the cardinality of $\mathcal{F}.$ Define $$
		\Ga:\cup_{N\in\mathcal{F}}P_N\to B(s)\cap\{x_d>0\},$$ by setting for $1\le j\le|\mathcal{F}|,$
		\begin{align*}
			\Ga|_{P_{N_j}}=\left(0,\cd,0,\frac{j}{n}\right)+\frac{1}{n^2}\Ga_{N_j}.
		\end{align*}
		Then $\Ga$ is a smooth embedding. Furthermore, for $1\le j\le |\mathcal{F}|$, the shortest distance between $\Ga(P_{N_j})$ and $\Ga(P_{N_{j+1}})$ is achieved by a line segment in $B(s)$ meeting both $\Ga(P_{N_j})$ and $\Ga(P_{N_{j+1}})$  only at the endpoints. Apply Lemma \ref{lembcsd} successively to obtain the boundary connected sum of $\Ga(P_{N_1})$, $\Ga(P_{N_2})$, $\cd,$ and reversing orientations of each connected component of the image of $\Ga$ if necessary. We are done.
	\end{proof}
	Let $U(\Si)$ be a closed tubular neighborhood of $\Si$ of small enough radius such that it does not intersect $\Ga(\prescript{}{\pd}{\#}_{N\in\mathcal{F}}P_N)$.
	By construction, the shortest distance between the smooth domain $\Ga(\prescript{}{\pd}{\#}_{N\in\mathcal{F}}P_N)$, and $U(\Si)$ is realized by a line segment in $B(s)$. Reversing orientations if necessary,
	\begin{fact}
		We obtain a boundary connected sum $U(\Si)\bcs\left(\prescript{}{\pd}{\#}_{N\in\mathcal{F}}\Ga(P_N)\right)$ embedded as a smooth open set of $M.$
	\end{fact}
	\subsection{A special neighborhood}
	In the notation of Lemma \ref{lemzla}, for each $N\in\mathcal{F}$, let use first consider $$\Si_{N_0}=X\times N_0+Y\times N_0.$$By constructions, $\reg(\Si_{N_0})$ has two connected components. Using Lemma \ref{lemcs}, we can define
	\begin{align*}
		\ov{\Si_{N_0}}=X\times N_0\#Y\times N_0,
	\end{align*}which has connected regular set and coincides with $\Si_{N_0}$ outside a region of an added neck away from $\sing \Si_{N_0}$.
	
	Use Lemma \ref{lemcs} successively in $U(\Si)\bcs\left(\prescript{}{\pd}{\#}_{N\in\mathcal{F}}\Ga(P_N)\right)$, we can obtain a integral current $\Si\#_{N\in\mathcal{F}}\ov{\Si_{N_0}},$ such that,
	\begin{fact}\label{fctsi}We have
		\begin{itemize}
			\item $\Si\#_{N\in\mathcal{F}}\ov{\Si_{N_0}}$ has connected regular set,
			\item $\Si\#_{N\in\mathcal{F}}\ov{\Si_{N_0}}$ differs from $\Si+\sum_{N\in\mathcal{F}}{\Si_{N_0}}$ only around $2|\mathcal{F}|$ added necks,
			\item $\sing\Si\#_{N\in\mathcal{F}}\ov{\Si_{N_0}}=\cup_{N\in\mathcal{F}}N$ and around each $N$, $\Si\#_{N\in\mathcal{F}}\ov{\Si_{N_0}}$ is a direct product of $N$ with two standard balls intersecting transversely inside the $\C^{d-\dim N}/\Z^{2(d-\dim N)}$ factor of the product $P_N.$
		\end{itemize} 
	\end{fact}
	Using the same proof as \cite[Lemma 5.0.1]{ZLns}, which actually dealt with much more complex singular sets, we can show that,
	\begin{fact}\label{fctu}
		There exists a smooth open set $U$ containing $\Si\#_{N\in\mathcal{F}}\ov{\Si_{N_0}}$ such that
		\begin{itemize}
			\item $U$ deformation retracts onto $\Si\#_{N\in\mathcal{F}}\ov{\Si_{N_0}}$,
			\item $U$ is contained in $U(\Si)\bcs\left(\prescript{}{\pd}{\#}_{N\in\mathcal{F}}\Ga(P_N)\right)$.
			\item $H_d(U,\Z)=\Z[\Si\#_{N\in\mathcal{F}}\ov{\Si_{N_0}}].$
		\end{itemize}
	\end{fact}
	Roughly speaking, the last bullet of Fact \ref{fctsi} makes it possible to find a triangulation of $M$ in which $\Si\#_{N\in\mathcal{F}}\ov{\Si_{N_0}}$ is realized as a subcomplex. Then standard topological arguments give the first two bullets. The connectedness of the regular set of $\Si\#_{N\in\mathcal{F}}\ov{\Si_{N_0}}$ ensures the last bullet. 
	\subsection{Adding singular sets}
	For the subsets $K_N\s N,$ use Lemma \ref{lemzs} to obtain a smooth function $f_N:N\to \R$ with $f_N\m(0)=K_N.$
	
	By multiplying $f_N$ with small enough real numbers, we can ensure that,
	\begin{fact}\label{factuu}
		$\Si_{N_{f_N}}=X\times N_0+Y\times N_{f_N}$	is fully contained in $U.$
	\end{fact}
	
	Again we can use Lemma \ref{lemcs} to obtain $$\ov{\Si_{N_{f_N}}}=X\times N_0\# Y\times N_{f_N}.$$ This time, using Lemma \ref{lemzla} and Lemma \ref{lemprod}. There exists a smooth metric $g_{P_N}$ on $P_N$ and a smooth calibration form $\phi_{P_N}$ on $P_N$ such that $\ov{\Si_{N_{f_N}}}$ is calibrated by $\phi_{P_N}$ in $g_{P_N}$ on $P_N.$
	
	Now use Lemma \ref{lemcs} again, we have
	\begin{fact}\label{fctf1}
		There exist a smooth metric $g$ and a smooth calibration form $\phi$ on $U(\Si)\bcs\left(\prescript{}{\pd}{\#}_{N\in\mathcal{F}}\Ga(P_N)\right)$, so that in metric $g$ the closed form $\phi$ calibrates,
		\begin{align*}
			T=\Si\#_{N\in\mathcal{F}}\ov{\Si_{N_{f_N}}}.
		\end{align*}
	\end{fact}Furthermore, we claim that
	\begin{fact}\label{fctf2}We have,
		\begin{itemize}
			\item $		T
			$ is the smoothly immersed image of the connected manifold $$\Si\#\left(\#_{N\in\mathcal{F}}(X\times N)\#(Y\times N)\right),$$
			\item the $\dim N$-th stratum in the Almgren stratification of $		T
			$ is $K_N.$
			\item  $		T
			$ is fully contained in $U$ and homologous to  $		\Si\#_{N\in\mathcal{F}}\ov{\Si_{N_0}}.$
		\end{itemize}
	\end{fact}
	\begin{proof}
		The first two bullets are immediate from the construction of connected sums. For the last bullet, containment in $U$ is due to Fact \ref{factuu} and Lemma \ref{lemcs}. As immersions, $T$ is a $C^\infty$ small perturbation of $		\Si\#_{N\in\mathcal{F}}\ov{\Si_{N_0}},$ so  $T$ is homologous to  $		\Si\#_{N\in\mathcal{F}}\ov{\Si_{N_0}}.$ 
	\end{proof}
	By Lemma \ref{lemzhang} applied to Fact \ref{fctf1} and Fact \ref{fctf2}, we are done.		\printbibliography
\end{document}